\documentclass[12pt]{article} 
\pdfoutput=1
\usepackage{graphicx}
\usepackage{amssymb, amsmath, amsthm, makeidx, amsthm}
 \usepackage[bookmarks=false]{hyperref}

\allowdisplaybreaks

\newenvironment{keywords}{
       \list{}{\advance\topsep by0.35cm\relax\small
       \leftmargin=1cm
       \labelwidth=0.35cm
       \listparindent=0.35cm
       \itemindent\listparindent
       \rightmargin\leftmargin}\item[\hskip\labelsep
                                     \bfseries Keywords:]}
     {\endlist}

\makeatletter
\renewcommand{\@biblabel}[1]{#1.} 
\makeatother

\newtheorem{Theorem}[equation]{Theorem}
\newtheorem{Proposition}[equation]{Proposition}

\newtheorem{Example}[equation]{Example}

\newtheorem*{rem}{Remark} 
\newenvironment{Remark}{\begin{rem}\normalfont}{\end{rem}}

\newtheorem*{rems}{Remarks} 

\newtheorem*{sol}{Solution}

\newtheorem*{nt}{Notes}

\makeatletter
\renewcommand{\@biblabel}[1]{#1.}
\makeatother

\newcommand{\suml}{\sum\limits}
\newcommand{\prodl}{\prod\limits}

\newcommand{\qrfac}[2]{{\left({#1}; q\right)_{#2}}} 
\newcommand{\pqrfac}[3]{{\left({#1};#3\right)_{#2}}}
\newcommand{\walker}{\vphantom{\prod\limits_1^{N^2}}}

\numberwithin{equation}{section}

\begin{document} 
\title{\sc In Praise of\\
 an
Elementary  
Identity 
of Euler}

\author{\sc Gaurav Bhatnagar
\\
Educomp Solutions Ltd.
\\
bhatnagarg@gmail.com
}


\date{June 7, 2011}

\maketitle 

\begin{center}
{\it Dedicated to S.~B.~Ekhad and D.~Zeilberger}
\end{center}

\begin{abstract}
We survey the applications of an elementary identity used by Euler in one of his proofs of the Pentagonal Number Theorem. Using a suitably reformulated version of this identity that we call Euler's Telescoping Lemma, we give alternate proofs of all the key summation theorems for terminating Hypergeometric Series and Basic Hypergeometric Series, including the terminating Binomial Theorem, the Chu--Vandermonde sum, the Pfaff--Saalsch\" utz sum, and their $q$-analogues. We also give a proof of Jackson's $q$-analog of Dougall's sum, the sum of a terminating, balanced, very-well-poised $_8\phi_7$ sum. Our proofs are conceptually the same  as those obtained by the WZ method, but done without using a computer. We survey identities for Generalized Hypergeometric Series given by Macdonald, and prove several identities for $q$-analogs of Fibonacci numbers and polynomials and Pell numbers that have appeared in combinatorial contexts. Some of these identities appear to be new. 
\end{abstract}


\section{Introduction}

One of the first results in $q$-series is Euler's 1740 expansion of the product
$$(1-q)(1-q^2)(1-q^3)(1-q^4)\cdots$$
into a power series in $q$. This expansion, known as Euler's Pentagonal Number Theorem is (for $|q|<1$):
\begin{align*}
(1-q)(1-q^2)&(1-q^3)\cdots\\
&= 1-q-q^2+q^5+q^7+\cdots\\
 &= 1 + \sum_{k=1}^{\infty} (-1)^k \left( q^{\frac{k(3k-1)}{2}} +
q^{\frac{k(3k+1)}{2}}\right).\\
\end{align*}
In his proof (explained by Andrews \cite{and1} and Bell \cite{bell}) 
Euler used the following elementary identity:
\begin{equation}\label{Eulerfunda}
(1+x_1)(1+x_2)(1+x_3)\cdots =
(1+x_1)+x_2(1+x_1)+x_3(1+x_1)(1+x_2)+\cdots
\end{equation}
Add the first two terms of the RHS to get $(1+x_1)(1+x_2)$ and then add that to the third term.  Continue in this manner. 

The objective of this paper is to demonstrate that this beautiful idea \eqref{Eulerfunda}, first used by Euler  more than 250 years ago,  can be used to prove many identities---of many kinds---in a unified manner.  As a demonstration of its power, we survey a wide selection of identities, all proved using \eqref{Eulerfunda}. 

We begin our survey of identities in \S\ref{sec:telescoping} with the Fibonacci identity
$$F_1+F_2+\cdots +F_n=F_{n+2}-1,$$
for the  sequence $F_n$  given by: $F_0=0$, $F_1=1$; and for $n\geq 0$,
$$F_{n+2}=F_{n+1}+F_{n}.$$
This famous identity (due to Lucas in 1876) is a well-known example of a telescoping sum. 

Indeed, a key idea of our work is that we recognize  \eqref{Eulerfunda} as a telescoping sum.   Then it is a small matter to show that any sum that telescopes is a special case of this elementary identity. Thus we have a characterization of a telescoping series that we call Euler's Telescoping Lemma (see \S\ref{sec:telescoping-lemma} and \S\ref{sec:telescoping-charac}). 

Another important sum that follows from Euler's Telescoping Lemma is the formula for the sum of the geometric sequence
$$1+x+x^2+\cdots+x^n=\frac{x^{n+1}-1}{x-1},$$
where $x\neq 1$.
Indeed, this formula is just the first in a set of summation theorems for the so-called {\em Hypergeometric series}, and their $q$-analogs, the {\em Basic Hypergeometric Series}. 
In \S\ref{sec:hyper}, \S\ref{sec:q-series} and \S\ref{sec:dougall}, we prove all the key terminating summation theorems for these series, beginning with the Binomial Theorem, and going up to Jackson's $q$-analog of Dougall's sum for a terminating, very-well-poised and balanced $_8\phi_7$ series (Gasper and Rahman \cite[eq.~(2.6.2)]{grhyp}).

In proving these identities, we use the WZ trick, an idea due to Wilf and Zeilberger~\cite{WZ90a} to rewrite a given sum in a way that it becomes a likely candidate for telescoping. Conceptually, our proofs are the same as those found by the WZ method. However, instead of using a computer,  we use Euler's idea to manually find the telescoping. For many examples, doing so is quite easy---almost as easy as using a computer. We call our method the EZ method, given that it rests on an application of Euler's idea and the WZ trick.  This method is outlined in \S\ref{sec:WZ}. 

Another example comes from Ramanujan (see Berndt \cite[Entry 25, p.~36]{berndt4}), who found a formula for the sum of $n+1$ terms of the series
$$\frac{1}{x+a_1}+\frac{a_1}{(x+a_1)(x+a_2)}+\frac{a_1a_2}{(x+a_1)(x+a_2)(x+a_3)}+\cdots.$$
Ramanujan's sum (given in \S\ref{sec:telescoping-charac}) has a general sequence as a parameter. In \S\ref{sec:generalized} we provide extensions of Ramanujan's identity due to Macdonald, the author of \cite{mac} (see Bhatnagar and Milne~\cite{gb-milne}). We call these series {\em Generalized Hypergeometric Series}. 

Finally, in \S\ref{sec:three-term}, we show that Euler's Telescoping Lemma is 
relevant even today by deriving identities for many combinatorial sequences.  These sequences satisfy a three-term recurrence relation, and are $q$-analogs of the Fibonacci or Pell sequences. We derive many identities found earlier by Andrews \cite{and3}, 
Garrett \cite{garrett},  Briggs, Little and Sellers \cite{bls}, Goyt and Mathisen \cite{goyt-mathisen} and others. 

In addition, we find many new identities for such sequences. In fact, we write down a general set of identities satisfied by all sequences that satisfy a three-term recurrence relation of the form:
$$x_{n+2}=a_nx_{n+1}+b_nx_n.$$
These identities are generalizations of  $6$ classical Fibonacci identities, that follow from Euler's identity. 

\section{Euler's Elementary Identity and Telescoping}\label{sec:telescoping}

In this section, we recognize Euler's elementary identity as a telescoping sum.  First note that the finite form of \eqref{Eulerfunda} is:
\begin{align}
 (1+x_1)(1+x_2)&\cdots (1+x_n)= \nonumber \\
(1+x_1)+x_2&(1+x_1)+\cdots+x_n(1+x_1)(1+x_2)\cdots(1+x_{n-1}).\label{eulerfinite}
\end{align}
To recognize \eqref{eulerfinite} as a telescoping sum, set $x_k\mapsto x_k-1$ to obtain:
\begin{align*}
x_1x_2\cdots x_n&=
x_1+(x_2-1)x_1+\cdots+(x_n-1)x_1x_2\cdots x_{n-1}\\
&=1+(x_1-1)+(x_2-1)x_1+\cdots+(x_n-1)x_1x_2\cdots x_{n-1}.
\end{align*}
We can write this as:
\begin{align*}
x_1x_2\cdots x_n -1 =
\sum_{k=1}^n (x_k-1)x_1x_2\cdots x_{k-1}.
\end{align*}
Here the product $x_1x_2\cdots x_{k-1}$ is considered to be equal to $1$ if $k=1$. 

It is now clear that the RHS telescopes.  

For applications, it is convenient to 
rewrite this by setting $x_k\mapsto u_k/v_k$ and $w_k=u_k-v_k$. In this manner, we obtain:
\begin{equation}\label{eulerfinite2}
\sum_{k=1}^n w_k \frac{u_1u_2\cdots u_{k-1}}{v_1v_2\cdots v_{k}}=
\frac{u_1u_2\cdots u_{n}}{v_1v_2\cdots v_{n}} -1,
\end{equation}
where $w_k=u_k-v_k$. 
\begin{Remark}
While it looks like \eqref{eulerfinite2}  has many more variables than \eqref{eulerfinite}, the two identities are equivalent. We can recover \eqref{eulerfinite} by setting $v_k=1$ and $u_k\mapsto x_k+1$ in \eqref{eulerfinite2}. 
\end{Remark}

\begin{Example}[Fibonacci Identities]\label{lucas}
Consider the Fibonacci Numbers defined as:
$F_0=0, F_1=1;$ and for $n\geq 0$,
$$F_{n+2}=F_{n+1}+F_{n}.$$
Then the following identities hold, for $n=0, 1, 2, \dots$: 
\begin{align}
\sum_{k=1}^n F_k &= F_{n+2}-1.\label{lucas1}\\
\sum_{k=1}^n F_{2k} &= F_{2n+1}-1.\label{lucaseven}\\
\sum_{k=1}^n F_{2k-1} &= F_{2n}.\label{lucasodd}\\
\sum_{k=1}^n F_{k}^2 &= F_{n}F_{n+1}.\label{lucas-square}\\
\sum_{k=1}^n (-1)^{k+1}F_{k+1} &= (-1)^{n-1}F_{n}.\label{lucas-alternating}\\
\sum_{k=1}^n \frac{F_{k-1}}{2^k} &= 1- \frac{F_{n+2}}{2^n}.\label{lucas-div-2n}
\end{align}
\end{Example}
\begin{Remark}
Some of these identities are due to Lucas and appear in Vajda \cite{vajda}. There are many more Fibonacci Identities that can be proved by telescoping and  are special cases of \eqref{eulerfinite2}. 
\end{Remark}
\begin{proof}
To prove the first identity, we set $u_k=F_{k+2},$ and  $v_k=F_{k+1}.$ Note that $w_k=F_{k+2}-F_{k+1}=F_k.$ Substituting in \eqref{eulerfinite2}, we obtain:
\begin{equation*}
\sum_{k=1}^n F_k\frac{F_3F_4\cdots F_{k+1}}{F_2F_3\cdots F_{k+1}} =
\frac{F_3F_4\cdots F_{n+2}}{F_2F_3\cdots F_{n+1}} -1,
\end{equation*}
or
\begin{equation*}
\sum_{k=1}^n F_k/F_2 =
F_{n+2}/{F_2} -1.
\end{equation*}
Since $F_2=1$, we immediately obtain \eqref{lucas1}. 

Next, set $u_k=F_{2k+1},$ and  $v_k=F_{2k-1}.$ Note that $w_k=F_{2k+1}-F_{2k-1}=F_{2k}+F_{2k-1}-F_{2k-1}=F_{2k}.$ Substituting in \eqref{eulerfinite2}, we obtain:
\begin{equation*}
\sum_{k=1}^n F_{2k}/F_1 =
F_{2n+1}/{F_1} -1.
\end{equation*}
Since $F_1=1$, we immediately get \eqref{lucaseven}.

Next, set $u_k=F_{2k+2},$ and  $v_k=F_{2k}.$ Note that $w_k=F_{2k+1}.$ Substituting in \eqref{eulerfinite2}, we obtain:
\begin{equation*}
\sum_{k=1}^n F_{2k+1}/F_2 =
F_{2n+2}/{F_2} -1.
\end{equation*}
Since $F_2=1$, we  get 
\begin{equation*}
\sum_{k=1}^n F_{2k+1} =
F_{2n+2} -1.
\end{equation*}
 Now note that:
\begin{align*}
\sum_{k=1}^n F_{2k+1} &=
F_{2n+2} -1\\
\implies 1+ \sum_{k=1}^n F_{2k+1} &=
F_{2n+2} \\
\implies \sum_{k=1}^{n+1} F_{2k-1} &=
F_{2n+2}.
\end{align*}
Identity \eqref{lucasodd} now follows by setting $n\mapsto n-1$. 

Next,  set $u_k=F_{k+1}F_{k+2},$ and  $v_k=F_{k}F_{k+1}.$ Note that $w_k=F_{k+1}^2.$ Substituting in \eqref{eulerfinite2}, we obtain:
\begin{equation*}
\sum_{k=1}^n F_{k+1}^2/F_1F_2 =
F_{n+1}F_{n+2}/{F_{1}F_2} -1.
\end{equation*}
Since $F_1=1=F_2$, we  get 
\begin{equation*}
\sum_{k=1}^n F_{k+1}^2 =
F_{n+1}F_{n+2} -1.
\end{equation*}
 Now note that:
\begin{align*}
\sum_{k=1}^n F_{k+1}^2 &=
F_{n+1}F_{n+2} -1\\
\implies 1+ \sum_{k=1}^n  F_{k+1}^2  &=
F_{n+1}F_{n+2} \\
\implies \sum_{k=1}^{n+1} F_{k}^2 &=
F_{n+1}F_{n+2}. 
\end{align*}
Identity \eqref{lucas-square} now follows by setting $n\mapsto n-1$. 

Next,  to obtain \eqref{lucas-alternating} set $u_k=F_{k+1},$ and  $v_k=-F_{k}.$ Note that $w_k=F_{k+2}.$ Substituting in \eqref{eulerfinite2}, we obtain:
\begin{equation*}
\sum_{k=1}^n (-1)^k F_{k+2}/F_1 =
(-1)^nF_{n+1}/{F_{1}} -1.
\end{equation*}
Now use $F_1=1$, and set $n\mapsto n-1$, to obtain \eqref{lucas-alternating}. Here too we need to make calculations such as in the proof of \eqref{lucas-square}. 

Finally,  to obtain \eqref{lucas-div-2n} set $u_k=F_{k+2},$ and  $v_k=2F_{k+1}.$ It is easy to show that $w_k=-F_{k-1}.$ Substituting in \eqref{eulerfinite2}, we obtain:
\begin{equation*}
\sum_{k=1}^n (-1)\frac{F_{k-1}}{2^k F_2} =
\frac{F_{n+2}}{2^nF_{2}} -1.
\end{equation*}
Now use $F_2=1$, and multiply both sides by $-1$ to obtain \eqref{lucas-div-2n}. 
 \end{proof}

\section{Euler's Telescoping Lemma}\label{sec:telescoping-lemma}
In this section, we write Euler's identity to fit the form of most identities. Specifically, we re-write Euler's identity so that:
\begin{itemize}
\item the index of summation of the sum in \eqref{eulerfinite2} ranges from $k=0$ to $n$.
\item the summand is $1$ for $k=0$.
\end{itemize}

For the first item, we need to define products as follows:
\begin{equation}\label{proddef}
\prod_{j=k}^m A_j = \
\begin{cases}
A_kA_{k+1}\cdots A_m  & {\text{ if } m\geq k},\\
1 & { \text{ if } m= k-1},\\
(A_{m+1}A_{m+2}\cdots A_{k-1})^{-1}  &{\text{ if } m\leq k-2}.\\
\end{cases}
\end{equation}
\begin{Remark}
This definition is motivated by our desire to ensure that

\begin{equation}\label{prod-motivate}
\left( \prod_{j=k}^{m-1} A_j\right) \times A_m = \prod_{j=k}^m A_j
\end{equation}
The reader should verify that if \eqref{prod-motivate} holds, then 
$$\prod_{j=1}^0 A_j=1,$$
and 
$$\prod_{j=1}^{-1} A_j=\frac{1}{A_0}.$$
These are consistent with  definition \eqref{proddef}.
\end{Remark}

For the second item, that is, to ensure the summand becomes $1$ when the index of summation $k$ is $0$ we multiply both sides of \eqref{eulerfinite2} by $u_0/w_0$. We also have to make a minor modification to the RHS of \eqref{eulerfinite2}.  In this manner,  we obtain:
\begin{Theorem}[Telescoping Lemma (Euler)] Let $u_k$, $v_k$ and $w_k$ be three sequences, such that
$$u_k-v_k=w_k.$$ 
Then we have:
\begin{equation}\label{telescoping-lemma}
\sum_{k=0}^n \frac{w_k}{w_0}\frac{u_0u_1\cdots u_{k-1}}{v_1v_2\cdots v_{k}}
= \frac{u_0}{w_0}\left( \frac{u_1u_2\cdots u_{n}}{v_1v_2\cdots v_{n}} - \frac{v_0}{u_0}\right),
\end{equation}
provided none of the denominators in \eqref{telescoping-lemma} are zero.
\end{Theorem}

\begin{proof}
Observe that: 
\begin{eqnarray*}
\sum_{k=0}^n \frac{w_k}{w_0}\frac{u_0u_1\cdots u_{k-1}}{v_1v_2\cdots v_{k}}
&=&
\sum_{k=0}^n \frac{u_0}{w_0}\left( \prod_{j=1}^{k}\frac{u_j}{v_j} - \prod_{j=1}^{k-1}\frac{u_j}{v_j}
\right)
\\
&=& \frac{u_0}{w_0}\left( \frac{u_1u_2\cdots u_{n}}{v_1v_2\cdots v_{n}} - \frac{v_0}{u_0}\right),
\end{eqnarray*}
by telescoping.
\end{proof}
Our next example is the sum of a Geometric sequence. 
\begin{Example}[Geometric Sum] For $x\neq 1$, we have:
$$\displaystyle \sum_{k=0}^n x^k =  \frac{x^{n+1}-1}{x-1}.$$
\end{Example}
\begin{proof}
Take $u_k=x$, and $v_k=1$ in \eqref{telescoping-lemma}. Then $w_k=x-1=w_0.$ We obtain
\begin{equation}\label{mac-geom}
\sum_{k=0}^n \frac{(x-1)}{(x-1)}\frac{x^k}{1}
= \frac{x}{x-1}\left( \frac{x^{n}}{1} - \frac{1}{x}\right).
 \end{equation}
 This gives us, on simplification:
 \begin{equation*}
\sum_{k=0}^n x^k
= \frac{x^{n+1}-1}{x-1}.
 \end{equation*}
 \end{proof}

In the form \eqref{telescoping-lemma} the Telescoping Lemma was used by Macdonald to generalize some results of Chu \cite{chu} (see \cite{gb-milne} and  \S \ref{sec:generalized}).   In the form \eqref{eulerfinite}, Euler's identity has been attributed to Schl\"{o}milch~\cite[pp.~26-31]{sch} by Gould (see \cite{gould-quan}).  
Another formulation was given by  Ramanujan (see Berndt \cite[Entry 26, eq.~(26.1), p.~27]{berndt4}).  Spiridonov \cite{spiridonov} mentions an equivalent formulation that the referee indicated is \lq\lq the general construction of telescoping sums". Indeed, we shall soon find that \eqref{telescoping-lemma} is a characterization of telescoping sums. 

\section{Telescoping Sums}\label{sec:telescoping-charac}

We now show that any sum that telescopes is a special case of Euler's Telescoping Lemma.  This is easy to see. 
A telescoping sum is of the form
$$f(k+1)-f(k)=a_k.$$
If we sum from $k=0$ to $n$, we obtain, by telescoping,
\begin{equation}\label{tel1}
f(n+1)-f(0) =\sum_{k=0}^n a_k.
\end{equation}
To recover \eqref{tel1} from the Telescoping Lemma \eqref{telescoping-lemma}, set
\begin{align*}
u_k&=f(k+1)\\
\text{and }v_k&=f(k).
\end{align*}
Thus, we have: $w_k=u_k-v_k =a_k$. Note that
$$u_0u_1\cdots u_{k-1}=v_1v_2\cdots v_k=f(1)f(2)\cdots f(k),$$
and \eqref{telescoping-lemma} yields
\begin{align*}
\sum_{k=0}^n \frac{a_k}{a_0} &= 
\frac{f(1)}{a_0}\left(\frac{f(2)f(3)\cdots f(n+1)}{f(1)f(2)\cdots f(n)}-\frac{f(0)}{f(1)}\right)\\
&= \frac{1}{a_0}\left(f(n+1)-f(0)\right).
\end{align*}
Multiplying both sides by $a_0$ we obtain \eqref{tel1} as required. 
\begin{Remark}
Since every telescoping sum is a special case of the Telescoping Lemma, we can argue that the Telescoping Lemma is a characterization of telescoping identities. So if we know that a sum telescopes, we can be sure it is a special case of \eqref{telescoping-lemma}. The many examples in this paper should convince the reader that this characterization is quite useful  in practice. 
\end{Remark}

Our next example is about a product that seems to be made for telescoping.
\begin{Example} We have, for $m=0, 1, 2, \dots $
\begin{equation}\label{rising-fact}
\sum_{k=1}^n k(k+1)\cdots (k+m-1) = \frac{1}{m+1}\left( n(n+1)\cdots (n+m)\right).
\end{equation}
\end{Example}
\begin{proof}
We take
\begin{align*}
u_k&=(k+1)(k+2)\cdots (k+m+1), \text{ and}\\
v_k=u_{k-1}&=k(k+1)\cdots (k+m).
\end{align*}
Then note that
\begin{align*}
w_k& =u_k-v_k = (m+1)(k+1)(k+2)\cdots (k+m),\\
\text{ and } w_0 &=(m+1)m!=(m+1)! .\\
\end{align*}
Substituting in \eqref{telescoping-lemma} now gives us:
\begin{equation*}
\sum_{k=0}^n \frac{(k+1)(k+2)\cdots (k+m)}{m!} = \frac{(n+1)(n+2)\cdots (n+m+1)}{m!}.
\end{equation*}
Note  that $v_0=0$, so the second term on the RHS of \eqref{telescoping-lemma} is $0$. 

Now multiplying both sides by $m!$, we obtain:
\begin{equation*}
\sum_{k=0}^n (k+1)(k+2)\cdots (k+m) = \frac{(n+1)(n+2)\cdots (n+m+1)}{m+1}.
\end{equation*}
Finally, we write the sum from $k=1$ to $n+1$ by replacing $k$ by $k-1$ in each term of the LHS, and then replace $n$ by $n-1$ to obtain \eqref{rising-fact}.
\end{proof}
\begin{Remark}
Set $m=1$ in \eqref{rising-fact} to obtain
\begin{equation*}
\sum_{k=1}^n k=\frac{1}{2}(n(n+1)),
\end{equation*}
the famous formula for the sum of the first $n$ natural numbers. 
\end{Remark}

The products appearing in the sum above are called {\bf rising factorials}. We use the following notation for the rising factorials
$$(x)_m :=
\begin{cases}
1   &{\text{ if } m=0}, \\
x(x+1)\cdots (x+m-1) &{\text{ if }} m\geq 1.\\
\end{cases}
$$

In this notation, the sum \eqref{rising-fact} becomes
\begin{equation}\label{rfacid1}
\sum_{k=1}^n (k)_m = \frac{1}{m+1}(n)_{m+1}.
\end{equation}
Another notation (given by \cite{concretemath}) used for rising factorials is 
$$x^{\overline{m}}:=(x)_m .$$
In this notation, the identity is even more suggestive:
\begin{equation*}
\sum_{k=1}^n k^{\overline{m}}= \frac{1}{m+1}n^{\overline{m+1}}.
\end{equation*}

The reader may enjoy proving (using the Telescoping Lemma) a similar identity where the rising factorials come in the denominator:
\begin{equation}
\sum_{k=1}^n \frac{1}{k(k+1)\cdots (k+m)} =\frac{1}{m}
\left( \frac{1}{m!} - \frac{1}{(n+1)(n+2)\cdots (n+m)}\right),
\end{equation}
for $m=1, 2, 3, \dots.$  This identity (for $m=1$) is used to show that the series $$\sum_{k=1}^{\infty} \frac{1}{k(k+1)}$$ converges. 

The next identity, due to Ramanujan, 
appeared in van der Poorten's charming exposition  \cite{vanderpoorten}  of  
Ap{\' e}ry's proof of the irrationality of $\zeta(3)$.

\begin{Example}[Ramanujan] Let $n$ be a non-negative integer and let $x$ and $a_k$  be such that the denominators in \eqref{ram1} are not zero. Then we have: 
\begin{align}
\sum_{k=0}^n & \frac{a_1a_2\cdots a_k}
{(x+a_1)(x+a_2)\cdots (x+a_{k+1})}\nonumber \\
&= \frac{1}{x}- \frac{a_1a_2\cdots a_{n+1}}
{x(x+a_1)(x+a_2)\cdots (x+a_{n+1})}.\label{ram1}
\end{align}
\end{Example}
\begin{Remark} Identity \eqref{ram1} is Entry 25 in Volume 4 of Ramanujan's Notebooks edited by Berndt~\cite[p.~36]{berndt4}, where it is proved by induction. The next result, Entry 26 \cite[eq.~(26.1), p.~27]{berndt4} is equivalent to Euler's Telescoping Lemma \eqref{telescoping-lemma}. 
\end{Remark}
\begin{proof}
We take
\begin{align*}
u_k&=a_{k+1} \text{ and}\\
v_k&=x+a_{k+1}.
\end{align*}
Then note that
$w_k =u_k-v_k = -x=w_0.$
Substituting in \eqref{telescoping-lemma} now gives us:
\begin{align*}
\sum_{k=0}^n & \frac{a_1a_2\cdots a_k}
{(x+a_2)(x+a_3)\cdots (x+a_{k+1})}\nonumber \\
&= -\frac{a_1}{x}\cdot \frac{a_2\cdots a_{n+1}}
{(x+a_2)(x+a_3)\cdots (x+a_{n+1})} + \frac{(x+a_1)}{x}.
\end{align*}
Now divide both sides by $(x+a_1)$ to obtain Ramanujan's identity. 
\end{proof}

Some extensions of Ramanujan's results appear in \S \ref{sec:generalized}.

\section{The WZ Trick and the EZ method}\label{sec:WZ}

We know that all telescoping sums are special cases of Euler's Telescoping Lemma. So if we know that an identity telescopes, then we can try to prove it by finding $u$ and $v$ such that the sum becomes a special case of the Telescoping Lemma.  But how do we know that a sum telescopes? It turns out that the sum telescopes for a large number of identities, once we apply a small trick of Wilf and Zeilberger \cite{WZ90a},
see \cite[ch.~7]{AeqB} or \cite[p.~166]{aar}. We call this trick the WZ trick. It is an important step of the WZ method given by Wilf and Zeilberger, described in \cite[ch.~7]{AeqB}. 

Suppose we have to prove a terminating identity in the form:
\begin{equation}\label{identity}
\sum_{k=0}^n LHS(n,k) = RHS(n).
\end{equation}
By dividing both sides by $RHS(n)$ we get an identity of the form:
\begin{equation}\label{identity1}
\sum_{k=0}^n F(n,k) = 1.
\end{equation}
Assume that the sum terminates naturally, that is, $F(n,k)=0$ if $k>n$. Then we have:
\begin{equation*}
\sum_{k=0}^{n+1} (F(n+1,k)-F(n,k)) = 0.
\end{equation*}
In the WZ method, we try to write this sum as a telescoping sum in $k$. That is, we try to find $G(n,k)$ such that:
\begin{equation}\label{zeil1}
F(n+1,k)-F(n,k) = G(n,k)-G(n,k-1).
\end{equation}
Now if by summing both sides over $k$, we find that
$$\sum_{k} (F(n+1,k)-F(n,k)) =0,$$
then we have the result:
$$\sum_{k} F(n,k) ={\text{ constant}}.$$
Finally, to  prove the identity \eqref{identity}, we need to verify the identity \eqref{identity1} for $n=0$.

Our approach is to fit the LHS of \eqref{zeil1} into the Telescoping Lemma.  The approach to prove any identity of the form \eqref{identity} is as follows:
\subsubsection*{The EZ Method to Prove Identities}
\begin{enumerate}
\item[Step 1.] Divide both sides by $RHS(n)$ to obtain \eqref{identity1}.
\item[Step 2.] Compute $F(n+1,k)-F(n,k)$.
\item[Step 3.] Find $u_k$ and $v_k$ such that the difference in Step 2 is (a multiple of) the summand on the LHS of \eqref{telescoping-lemma}.
\item[Step 4.] Sum over $k$ using the Telescoping Lemma, and verify the sum is $0$. In all the examples below, this happens because $u_{n+1}=0$ and $v_0=0$.
\item[Step 5.] Step 4 shows that the sum in \eqref{identity1} is a constant. Verify that the sum is $1$ at $n=0$ to finish the proof. 
\end{enumerate}

The following example should help the reader understand the method. This is an identity that follows from the Binomial Theorem and is the simplest possible example that shows all the interesting features of this method. 
\begin{Example} Let $n=0, 1, 2, 3, \dots.$ Then we have:
\begin{equation}\label{bin-ex1}
\sum_{k=0}^n \frac{(-1)^k (-n)_k}{k!} = 2^n.
\end{equation}
\end{Example}
\begin{proof}
The first step is to divide both sides by $2^n$ to obtain
\begin{equation*}\label{bin-ex2}
\sum_{k=0}^n \frac{(-1)^k (-n)_k}{2^n k!}=1
\end{equation*}
Now let
$$F(n,k)=\frac{(-1)^k(-n)_k}{2^n k!}.$$
We compute $F(n+1,k)-F(n,k)$ to obtain:
\begin{align*}
F(n+1,k)-F(n,k) &= \frac{(-1)^k}{k!}\left( \frac{(-n-1)_k}{2^{n+1}} -\frac{(-n)_k}{2^{n}}\right)\\
&= \frac{(-1)^k}{2^{n+1}k!}\left( (-n-1)_k -2(-n)_k\right).
\end{align*}
Now note that
\begin{align*}
 (-n-1)_k -2(-n)_k &\\
 = (-n)(-n+1)&\cdots (-n+k-2)\left( -n-1-2(-n+k-1)\right)\\
&= (-1)\frac{n-2k+1}{n+1} (-n-1)_k.
\end{align*}
Thus we have:
\begin{equation*}
F(n+1,k)-F(n,k) = \frac{(-1)}{2^{n+1}}\left(  \frac{n-2k+1}{n+1} \frac{(-1)^k (-n-1)_k}{k!}\right).
\end{equation*}
We have now reached a stage where we can apply the Telescoping Lemma \eqref{telescoping-lemma}. 
We set
\begin{align*}
u_k&=(-1)(-n-1+k) = (n-k+1)\\
\text{and } v_k & = k
\end{align*}
We find that $w_k=u_k-v_k=n-2k+1$ and $w_0=n+1$. Thus we have, with $u_k$, $v_k$ and $w_k$ as above:
\begin{align*}
\sum_{k=0}^{n+1}\left( F(n+1,k)-F(n,k)\right) &= \frac{(-1)}{2^{n+1}}
\sum_{k=0}^{n+1} \frac{w_k}{w_0}\frac{u_0u_1\cdots u_{k-1}}{v_1v_2\cdots v_{k}}
\\
&= \frac{(-1)}{2^{n+1}}
\frac{u_0}{w_0}\left( \frac{u_1u_2\cdots u_{n+1}}{v_1v_2\cdots v_{n+1}} - \frac{v_0}{u_0}\right).
\end{align*}
Now note that
$u_{n+1}=0$ and $v_0=0$, so the RHS is $0$.
Thus we have 
\begin{equation*}
\sum_{k=0}^{n+1}\left( F(n+1,k)-F(n,k)\right) = 0.
\end{equation*}
Thus the sum 
$$\sum_{k=0}^n F(n,k)$$ is a constant. We now verify that the sum is $1$ for $n=0$, and so the constant is $1$.  This completes the proof of \eqref{bin-ex1}.
\end{proof}
\begin{Remark} 
Note that  the numerator of the summand in \eqref{bin-ex1} has the factor 
$$(-n)_k=(-n)(-n+1)\cdots (-n+k-1).$$ This factor makes the sum terminate naturally, that is, when $k>n$, the terms of the sum become $0$.
 
Corresponding to this factor,  we have a factor $u_k=-n-1+k$. Note that $u_{n+1}=0$ which makes the first term of the RHS of \eqref{telescoping-lemma} become $0$, when $k=n+1$. 

Similarly, note the factor $k!$ in the denominator of the summand in \eqref{bin-ex1}. If we view it as
$$\frac{1}{k!}=\frac{1}{\Gamma(k+1)}$$
then we see that this factor is $0$ when $k$ is a negative integer. It makes the sum terminate naturally from below.

Corresponding to this factor we have $v_k=k$. Thus $v_0=0$, and this ensures the second term of the RHS of  \eqref{telescoping-lemma} is $0$.

In many of our examples,  $u_k$ has the factor  $-n-1+k$ and $v_k$ the factor $k$. 
\end{Remark}
\begin{Remark}
When is it a good idea to apply \eqref{telescoping-lemma} directly, without using the WZ trick? We try \eqref{telescoping-lemma} directly when the sum does not have a factor that terminates the sum naturally. These kind of sums are called {\em indefinite} sums. The geometric sum is an example of such a sum. 
\end{Remark}

\section{Examples of Hypergeometric Identities}\label{sec:hyper}
In this section, we give more examples  from the theory of Hypergeometric series.  Hypergeometric series are of the form $\sum t_k$, where $t_{k+1}/t_k$ is a rational function of $k$. (The geometric sum is where this ratio is a constant.) Most special functions and binomial coefficient identities are examples of such series. All the identities here are proved using the EZ method described in \S\ref{sec:WZ}, by the WZ trick followed by an application of Euler's Telescoping Lemma. 

The reader will find it useful to compare the proofs of examples in this section with those in \S 3.11 and \S 3.12 of Andrews, Askey and Roy \cite{aar}.

The first example is:
\begin{Example}[The Binomial Theorem]\label{ex-binomial} Let $n$ be a non-negative integer. Then we have
\begin{equation}\label{binomial}
\sum_{k=0}^n {n \choose k} x^k =(1+x)^n.
\end{equation}
\end{Example}
\begin{proof}
We will show:
\begin{equation}\label{binomial1}
\sum_{k=0}^n \frac{1}{(1+x)^n}{n \choose k} x^k =1.
\end{equation}
Let 
\begin{align*}
F(n,k)&= \frac{x^k}{(1+x)^n}{n \choose k}\\
&=\frac{x^k}{(1+x)^n}\frac{n(n-1)(n-2)\cdots (n-k+1)}{k!}.
\end{align*}
We find that
\begin{align*}
F(n+1,k)-F(n,k) = \frac{(-1)x}{(1+x)^{n+1}}\\
\left( \frac{x(n-k+1)-k}{x(n+1)}\right. & 
\left.
\frac{x^k(n+1)n(n-1)\cdots (n-k+2)}{k!}\right).
\end{align*}
Next  we compare the expression in the brackets with the summand in \eqref{telescoping-lemma}. 
It is easy to see that the following will do the trick:
\begin{align*}
u_k&=x(n-k+1)\\
\text{and } v_k & = k.
\end{align*}
We find that $w_k=u_k-v_k=(x(n-k+1)-k)$ and $w_0=x(n+1)$. 
Note further that $u_{n+1}=0$ and $v_0=0$. 

Thus we have, with $u_k$, $v_k$ and $w_k$ as above:
\begin{align*}
\sum_{k=0}^{n+1}\left( F(n+1,k)-F(n,k)\right) &= \frac{(-1)x}{(1+x)^{n+1}}
\sum_{k=0}^{n+1} \frac{w_k}{w_0}\frac{u_0u_1\cdots u_{k-1}}{v_1v_2\cdots v_{k}}
\\
&= 0.
\end{align*}
The RHS is $0$ because $u_{n+1}=0$ and $v_0=0$.
Thus $$\sum_{k=0}^{n} F(n,k)$$
is a constant. To finish the proof, we verify that this sum is $1$ when $n=0$. 
\end{proof}
\begin{Remark}
One can write binomial coefficient identities using rising factorials. To rewrite \eqref{binomial}, note that
\begin{align*}
{n\choose k}&=\frac{n(n-1)(n-2)\cdots (n-k+1)}{k!}\\
&=(-1)^k\frac{(-n)(-n+1)(-n+2)\cdots (-n+k-1)}{k!}\\
&=(-1)^k\frac{(-n)_k}{k!}.
\end{align*}
Thus we can write \eqref{binomial} as
\begin{equation}\label{binomial2}
\sum_{k=0}^n  (-1)^k \frac{(-n)_k}{k!} x^k =(1+x)^n.
\end{equation}
 Set $x=1$ to recover \eqref{bin-ex1}.  
\end{Remark}

Our next example has two more parameters.
\begin{Example}[Chu (1303)--Vandermonde (1772)]\label{ex-chu-vandermonde}
Let $n$ be a non-negative integer and let $a$ and $b$ be such that the denominators in \eqref{chu-vandermonde} are not zero. Then we have:
\begin{equation}\label{chu-vandermonde}
\sum_{k=0}^n \frac{(a)_k}{(b)_k} \frac{(-n)_k}{k!} = \frac{(b-a)_n}{(b)_n}. 
\end{equation}
\end{Example}
\begin{Remark} The reader is referred to Andrews, Askey and Roy \cite[Cor.~2.2.3]{aar} for the history of the Chu--Vandermonde identity.
\end{Remark}
\begin{proof}
Again, by dividing by the RHS, we form an equivalent identity of the form
\begin{equation*}
\sum_{k=0}^n F(n,k) = 1,
\end{equation*}
where $F(n,k)$ is defined as:
\begin{equation*}
F(n,k)= \frac{(b)_n}{(b-a)_n}
 \frac{(a)_k}{(b)_k} \frac{(-n)_k}{k!}.
\end{equation*}
We find that
\begin{align*}
F(n+1,k)-F(n,k) = \frac{a(b)_n}{(b-a)_{n+1}}\\
\left( \frac{a(n-k+1)+k(b+n)}{a(n+1)}\right. & 
\left.
\frac{(a)_k(-n-1)_k}{(b)_k k!}\right).
\end{align*}
Next  we compare the expression in the brackets with the summand in \eqref{telescoping-lemma}. 
It is easy to see that the following will do the trick:
\begin{align*}
u_k&=(a+k)(-n-1+k)\\
\text{and } v_k & = k (b+k-1).
\end{align*}
We find that $w_k=u_k-v_k=(-1)(a(n-k+1)+k(b+n))$ and $w_0=-a(n+1)$. 
Note further that $u_{n+1}=0$ and $v_0=0$. 

Thus we have, with $u_k$, $v_k$ and $w_k$ as above:

\begin{align*}
\sum_{k=0}^{n+1}\left( F(n+1,k)-F(n,k)\right) &=  \frac{a(b)_n}{(b-a)_{n+1}}
\sum_{k=0}^{n+1} \frac{w_k}{w_0}\frac{u_0u_1\cdots u_{k-1}}{v_1v_2\cdots v_{k}}
\\
&= 0.
\end{align*}
Thus $$\sum_{k=0}^{n} F(n,k)$$
is a constant. To finish the proof, we verify that this sum is $1$ when $n=0$. 
\end{proof}

Our next example is a sum discovered independently by Pfaff (1797) and Saalsch\" utz (1890), see \cite[eq.~(1.7.2)]{grhyp}. 
\begin{Example}[Pfaff--Saalsch\" utz Theorem]\label{ex-pfaff-saalchutz} 
Let $n$ be a non-negative integer and let $a$, $b$ and $c$ be such that the denominators in \eqref{pfaff-saalchutz} are not zero. Then we have:
\begin{equation}\label{pfaff-saalchutz}
\sum_{k=0}^n \frac{(a)_k (b)_k}{(c)_k (1-n+a+b-c)_k } \frac{(-n)_k}{k!} = \frac{(c-a)_n (c-b)_n}
{(c)_n(c-a-b)_n}. 
\end{equation}
\end{Example}
\begin{Remark} The reader is referred to Andrews, Askey and Roy \cite[eq.~(2.2.8)]{aar} for historical remarks regarding the Pfaff--Saalsch\" utz Theorem. See Andrews \cite{and-pfaff2} for Pfaff's own (and perhaps the simplest) proof of \eqref{pfaff-saalchutz}.
\end{Remark}
\begin{proof}
Again, by dividing by the RHS, we form an equivalent identity of the form
\begin{equation*}
\sum_{k=0}^n F(n,k) = 1,
\end{equation*}
where $F(n,k)$ is defined as:
\begin{equation*}
F(n,k)= \frac{(c)_n(c-a-b)_n}{(c-a)_n(c-b)_n}
 \frac{(a)_k(b)_k}{(c)_k(1-n+a+b-c)_k} \frac{(-n)_k}{k!}.
\end{equation*}

We find that
\begin{align*}
F(n+1,k)-F(n,k) = &\frac{(c)_n(c-a-b)_n}{(c-a)_{n+1}(c-b)_{n+1}}\\
&\times \left( \frac{(a)_k(b)_k(-n)_{k-1}}{(c)_k   (1-n+a+b-c)_k k!}\right)\\
&\times \left( (c+n) (-n+k+a+b-c)(n+1)\right.\\
&\left. -(c-a+n)(c-b+n)(-n+k-1)  \right)\\
&= (-1)ab\frac{(c)_n(c-a-b)_n}{(c-a)_{n+1}(c-b)_{n+1}}\\
&\times \left( \frac{(a)_k(b)_k(-n-1)_{k}}{(c)_k (1-n+a+b-c)_k  k! }\right)\\
&\times \left(\frac{(c+n)(a+b-c+1)k +ab(n-k+1)}{ab(n+1)}\right).
\end{align*}
Next  we compare this expression with the summand in \eqref{telescoping-lemma}. 
It is easy to see that the following will do the trick:
\begin{align*}
u_k&=(a+k)(b+k)(-n-1+k)\\
\text{and } v_k & = k (c+k-1)(-n+k+a+b-c).
\end{align*}
We find that 
\begin{align*}
w_k=&
(a+k)(b+k)(-n-1+k)-k (c+k-1)(-n+k+a+b-c)\\
&= -\left((c+n)(a+b-c+1)k +ab(n-k+1)\right)
\end{align*}
 and $w_0=-ab(n+1)$. 
Note further that $u_{n+1}=0$ and $v_0=0$. 

Thus we have, with $u_k$, $v_k$ and $w_k$ as above:

\begin{align*}
\sum_{k=0}^{n+1}\left( F(n+1,k)-F(n,k)\right) &=  
(-1)ab\frac{(c)_n(c-a-b)_n}{(c-a)_{n+1}(c-b)_{n+1}} \\
&\times
\sum_{k=0}^{n+1} \frac{w_k}{w_0}\frac{u_0u_1\cdots u_{k-1}}{v_1v_2\cdots v_{k}}
\\
&= 0.
\end{align*}
Thus $$\sum_{k=0}^{n} F(n,k)$$
is a constant. Now verify that this sum is $1$ when $n=0$. 
This finishes the proof of the Pfaff--Saalsch\" {u}tz identity.
\end{proof}

\begin{Remark}
The only difficult part of our proof of the Pfaff--Saalsch\" utz summation is the algebra required to prove that:
\begin{align*}
&(c+n) (-n+k+a+b-c)(n+1)
 -(c-a+n)(c-b+n)(-n+k-1)\\
&=
(c+n)(a+b-c+1)k +ab(n-k+1)\\
&=
-\left( (a+k)(b+k)(-n-1+k)-k (c+k-1)(-n+k+a+b-c)\right). \\
\end{align*}
The first equality is required to simplify $F(n+1,k)-F(n,k)$; the second to find an expression for $w_k=u_k-v_k.$

Compare these with the corresponding calculations from the proof of the Chu--Vander\-monde Identity:
\begin{align*}
&(b+n) (n+1)
 +(b-a+n)(-n+k-1)\\
&=
\left((b+n)k +a(n-k+1)\right)\\
&=
\left( (a+k)(-n-1+k)-k(b+k-1)\right). \\
\end{align*}
The calculations in the proof of the Binomial Theorem are even simpler. 
 Upon examining these expressions it is apparent that there is a natural hierarchy both in the identities and their proofs. See also our work in \S \ref{sec:dougall}   that sheds some light on these calculations. 
\end{Remark}

\begin{Remark}
For a large number of identities, the WZ method produces a \lq\lq certificate" $R(n,k)$ such that
$G(n,k)=R(n,k)F(n,k)$, where $G(n,k)$ satisfies \eqref{zeil1}. In our case, we directly produce $G(n,k)$ by appealing to the Telescoping Lemma. 

Note that in all the examples, we ended up with an expression of the form
$$F(n+1,k)-F(n,k)=g(n)\left( 
\frac{w_k}{w_0}\frac{u_0u_1\cdots u_{k-1}}{v_1v_2\cdots v_{k}}\right),
$$
where $w_k=u_k-v_k$. 
Thus we have:
$$G(n,k)= g(n) \frac{u_0}{u_0-v_0} \frac{u_1u_2\cdots u_{k}}{v_1v_2\cdots v_{k}}.$$
The WZ methodology is more general than the EZ method---it works even when a relation of the form \eqref{zeil1} does not apply, when the LHS of \eqref{zeil1} is more complicated. 
\end{Remark}

The proofs given in this section should be compared with corresponding proofs by Pfaff's method described by Andrews \cite{and-pfaff2} and Andrews, Askey and Roy \cite[\S 3.11--\S 3.12]{aar}. In Pfaff's method, one does not divide by  $RHS(n)$ in \eqref{identity}.  However, for many examples, computing $F(n+1,k)-F(n,k)$ leads to a three term recurrence relation that determines the sum. To complete the proof, we show that the product side too satisfies the same relation. When this procedure works, then the calculations are even simpler than the corresponding calculations of the EZ method. The reader may also consult Guo and Zeng \cite{guo-zeng} for a similar method.

\section{Examples from $\boldsymbol q$-series}\label{sec:q-series}

The examples we consider in this section are the $q$-analogs of the corresponding sums in \S\ref{sec:hyper}. The proofs are analogous to those in the Hypergeometric case, and follow the EZ method outlined in \S\ref{sec:WZ}. 

We define the $q$-rising factorial (for $q$ a complex number) as the product:
$$\qrfac{a}{m} :=
\begin{cases}
1 &{\text{ if } m=0},\\
(1-a)(1-aq)\cdots (1-aq^{m-1}) &{\text{ if }} m\geq 1.\\
\end{cases}
$$
The limit
$$\lim_{q\to 1}\frac{1-q^A}{1-q}=A,$$
implies:
$$\lim_{q\to 1}\frac{\qrfac{q^a}{m}}{(1-q)^m}=(a)_m.$$
This motivates the use of the term \lq $q$-analog\rq\  for these sums. 

\begin{Example}[The terminating $q$-binomial sum] Let $n$ be a non-negative integer and $q$ a complex number such that the denominator in \eqref{qbinomial1} is not zero. Then we have: 
\begin{equation}\label{qbinomial1}
\sum_{k=0}^n  \frac{\qrfac{q^{-n}}{k}}{\qrfac{q}{k}} (zq^n)^k =\qrfac{z}{n}.
\end{equation}
\end{Example}
\begin{Remark} The terminating $q$-binomial theorem may be found in Gasper and Rahman \cite[eq.~(II.4)]{grhyp} where we take $z\mapsto zq^n$.  If we take the limit as $q\to 1$ in \eqref{qbinomial1}, and set $z\mapsto -x$, we obtain the Binomial Theorem \eqref{binomial2}.  
\end{Remark}
\begin{proof}
By dividing by the RHS, we form an equivalent identity of the form
\begin{equation*}
\sum_{k=0}^n F(n,k) = 1,
\end{equation*}
where $F(n,k)$ is defined as:
\begin{equation*}
F(n,k)= \frac{1}{\qrfac{z}{n}}\frac{\qrfac{q^{-n}}{k}}{\qrfac{q}{k}}
\left(zq^n\right)^k.
\end{equation*}
We find that
\begin{align*}
F(n+1,k)-F(n,k) &= \frac{z^kq^{nk}}{\qrfac{q}{k}\qrfac{z}{n+1}}
\frac{\qrfac{q^{-n-1}}{k}}{1-q^{-n-1}}
\\
&\left( \left(1-q^{-n-1}\right)q^k
- \left(1-q^{-n+k-1}\right)(1-zq^n)\right)\\
&= \frac{zq^n}{\qrfac{z}{n+1}}
\frac{z^{k}q^{nk}\qrfac{q^{-n-1}}{k}}{\qrfac{q}{k}}
\\
&
\left(\frac{ 
 zq^n\left(1-q^{-n+k-1}\right)
-\left(1-q^{k}\right)}
{zq^n\left(1-q^{-n-1}\right)}
\right)
.\\
\end{align*}
Next  we compare the expression  above with the summand in \eqref{telescoping-lemma}. 
It is easy to see that the following will do the trick:
\begin{align*}
u_k&=zq^n\left( 1-q^{-n-1+k}\right)\\
\text{and } v_k & = 1-q^k.
\end{align*}
We find that 
\begin{align*}
w_k=u_k-v_k &=
 zq^n\left(1-q^{-n+k-1}\right)
-\left(1-q^{k}\right)\\
\text{and }
w_0 &=zq^n\left(1-q^{-n-1}\right).
\end{align*} 
Note further that $u_{n+1}=0$ and $v_0=0$. 

Thus we have, with $u_k$, $v_k$ and $w_k$ as above:
\begin{align*}
\sum_{k=0}^{n+1}\left( F(n+1,k)-F(n,k)\right) &= \frac{zq^n}{\qrfac{z}{n+1}}
\sum_{k=0}^{n+1} \frac{w_k}{w_0}\frac{u_0u_1\cdots u_{k-1}}{v_1v_2\cdots v_{k}}
\\
&= 0.
\end{align*}
Thus $$\sum_{k=0}^{n} F(n,k)$$
is a constant. To finish the proof, we verify that this sum is $1$ when $n=0$. 
\end{proof}
\begin{Remark}
Compare the statement of the $q$-Binomial Theorem with 
\eqref{binomial2}. The proof is analogous to that of Example \ref{ex-binomial}.
\end{Remark}
\begin{Remark}
Note that $u_k$ has the factor $\left(1-q^{-n+k-1}\right)$ which corresponds to the factor $-n+k-1$ in our examples in the previous section. This makes $u_{n+1}=0$. Similarly, $v_k$ has the factor $\left(1-q^{k}\right)$, that corresponds to the $k$ and we have $v_0=0$. 
\end{Remark}

\begin{Example}[A $q$-Chu-Vandermonde sum] Let $n$ be a non-negative integer and let $q$, $a$ and $b$ be such that the denominators in \eqref{qchu-vandermonde1} are not zero. Then we have:
\begin{equation}\label{qchu-vandermonde1}
\sum_{k=0}^n  \frac{\qrfac{a}{k}\qrfac{q^{-n}}{k}}{\qrfac{b}{k} \qrfac{q}{k}} \left(\frac{bq^n}{a}\right)^k =
\frac{\qrfac{b/a}{n}}{\qrfac{b}{n}}.
\end{equation}
\end{Example}
\begin{Remark} Identity \eqref{qchu-vandermonde1} is one of the two $q$-analogs of the Chu--Vandermonde identity \eqref{chu-vandermonde}. See Gasper and Rahman \cite[eq. (1.5.2)]{grhyp}.
\end{Remark}
\begin{proof}
By dividing by the RHS, we form an equivalent identity of the form
\begin{equation*}
\sum_{k=0}^n F(n,k) = 1,
\end{equation*}
where $F(n,k)$ is defined as:
\begin{equation*}
F(n,k)= \frac{\qrfac{b}{n}}{\qrfac{b/a}{n}}\frac{\qrfac{a}{k}\qrfac{q^{-n}}{k}}{\qrfac{b}{k}\qrfac{q}{k}}
\left(\frac{bq^n}{a}\right)^k.
\end{equation*}
We find that
\begin{align*}
F(n+1,k)-&F(n,k) = 
\frac{\qrfac{a}{k}}{\qrfac{b}{k}\qrfac{q}{k}}\left(\frac{b}{a}\right)^k \cdot
\frac{\qrfac{b}{n}}{\qrfac{b/a}{n+1}}
\frac{\qrfac{q^{-n-1}}{k}}{1-q^{-n-1}}q^{nk}
\\
&\left( \left( 1-bq^n\right)\left(1-q^{-n-1}\right)q^k
- \left( 1-bq^n/a\right) \left(1-q^{-n+k-1}\right)\right)\\
&=
\frac{b(1-a)q^n\qrfac{b}{n}}{a\qrfac{b/a}{n+1}}\cdot
\frac{\qrfac{a}{k}\qrfac{q^{-n-1}}{k}}{\qrfac{b}{k}\qrfac{q}{k}}\left(\frac{bq^n}{a}\right)^k
\\
&
\left(\frac{ 
  \frac{b}{a}(1-a)q^n  \left(1-q^{-n+k-1}\right)
-\left(1-bq^n\right)\left(1-q^{k}\right)}
{\frac{b}{a}(1-a)q^n\left(1-q^{-n-1}\right)}
\right)
.\\
\end{align*}
Next  we compare the expression  above with the summand in \eqref{telescoping-lemma}. 
It is easy to see that the following will do the trick:
\begin{align*}
u_k&=\left(1-aq^k\right)\left( 1-q^{-n-1+k}\right)bq^n/a\\
\text{and } v_k & = \left(1-bq^{k-1}\right)\left(1-q^k\right).
\end{align*}
We find that 
\begin{align*}
w_k=u_k-v_k &=
 \left(1-aq^k\right)\left( 1-q^{-n-1+k}\right)bq^n/a
 -
  \left(1-bq^{k-1}\right)\left(1-q^k\right)\\
  &=
  \frac{b}{a}(1-a)q^n  \left(1-q^{-n+k-1}\right)
-\left(1-bq^n\right)\left(1-q^{k}\right),
   \\
\text{and }
w_0 &=\frac{b}{a}(1-a)q^n\left(1-q^{-n-1}\right).
\end{align*} 
Note further that $u_{n+1}=0$ and $v_0=0$. 

Thus we have, with $u_k$, $v_k$ and $w_k$ as above:
\begin{align*}
\sum_{k=0}^{n+1}\left( F(n+1,k)-F(n,k)\right) &= 
\frac{b(1-a)q^n\qrfac{b}{n}}{a\qrfac{b/a}{n+1}}
\sum_{k=0}^{n+1} \frac{w_k}{w_0}\frac{u_0u_1\cdots u_{k-1}}{v_1v_2\cdots v_{k}}
\\
&= 0.
\end{align*}
Thus $$\sum_{k=0}^{n} F(n,k)$$
is a constant. To finish the proof, we verify that this sum is $1$ when $n=0$. 
\end{proof}

\begin{Example}[The $q$-Pfaff--Saalsch\" utz sum (Jackson (1910))] Let $n$ be a non-negative integer and let $q$, $a$,  $b$ and $c$ be such that the denominators in \eqref{qpfaff-saalchutz1} are not zero. Then we have:
\begin{equation}\label{qpfaff-saalchutz1}
\sum_{k=0}^n  \frac{\qrfac{a}{k}\qrfac{b}{k}\qrfac{q^{-n}}{k}}{\qrfac{c}{k}\qrfac{abq^{1-n}/c}{k} \qrfac{q}{k}} q^k =
\frac{\qrfac{c/a}{n}\qrfac{c/b}{n}}{\qrfac{c}{n}\qrfac{c/ab}{n}}.
\end{equation}
\end{Example}
\begin{Remark} Identity \eqref{qpfaff-saalchutz1} is the $q$-analog of the Pfaff--Saalsch\" utz sum  \eqref{pfaff-saalchutz}, see Gasper and Rahman \cite[eq. (1.7.2)]{grhyp}.
\end{Remark}
\begin{proof}
By dividing by the RHS, we form an equivalent identity of the form
\begin{equation*}
\sum_{k=0}^n F(n,k) = 1,
\end{equation*}
where $F(n,k)$ is defined as:
\begin{equation*}
F(n,k)= \frac{\qrfac{c}{n}\qrfac{c/ab}{n}}{\qrfac{c/a}{n}\qrfac{c/b}{n}}
\frac{\qrfac{a}{k}\qrfac{b}{k}\qrfac{q^{-n}}{k}}{\qrfac{c}{k}\qrfac{abq^{1-n}/c}{k}\qrfac{q}{k}}
q^k.
\end{equation*}
We find that
\begin{align*}
F(n+1,k)-F(n,k) &= 
\frac{\qrfac{a}{k}\qrfac{b}{k}}{\qrfac{c}{k}\qrfac{q}{k}}q^k \cdot
\frac{\qrfac{c}{n}\qrfac{c/ab}{n}}{\qrfac{c/a}{n+1}\qrfac{c/b}{n+1}}
\\
\times &\frac{\qrfac{q^{-n-1}}{k}}{\left(1-q^{-n-1}\right)}
\frac{1}{\qrfac{abq^{1-n}/c}{k}}(-1)
\\
&\left( \left( 1-cq^n\right)\left(1-abq^{-n+k}/c\right) \left(1-q^{-n-1}\right)\left( cq^n/ab\right)
\right.
\\
&+ \left. \left( 1-cq^n/a\right)\left( 1-cq^n/b\right) \left(1-q^{-n+k-1}\right)\right).
\end{align*}
We can show that the expression in the brackets equals:
\begin{align*}
&\left( \left( 1-cq^n\right) \left(1-abq^{-n+k}/c\right) \left(1-q^{-n-1}\right)
\left( cq^n/ab\right)
\right.
\\
&+ \left. \left( 1-cq^n/a\right)\left( 1-cq^n/b\right) \left(1-q^{-n+k-1}\right)\right)\\
&=\left(
\left(1-cq^n\right)\left( 1-c/abq \right) \left(1-q^{k}\right) \right.\\
&+ \left. \frac{cq^n}{ab}(1-a)(1-b) \left(1-q^{-n+k-1}\right) 
\right) 
.
\end{align*}

Next  we compare the expression  above with the summand in \eqref{telescoping-lemma}. We take:
\begin{align*}
u_k&=\left(1-aq^k\right)\left(1-bq^k\right)\left( 1-q^{-n-1+k}\right)\\
\text{and } v_k & = \left(1-cq^{k-1}\right)\left(1-abq^{-n+k}/c\right)\left(1-q^k\right).
\end{align*}
We can show that:
\begin{align*}
w_k=u_k-v_k &=
 \left(abq^{-n+k}/c\right)
 \left[\walker
\left(1-cq^n\right)\left( 1-c/abq \right) \left(1-q^{k}\right) \right.\\
&+ \left. \frac{cq^n}{ab}(1-a)(1-b) \left(1-q^{-n+k-1}\right) 
\walker \right],
   \\
\text{and }
w_0 &=(1-a)(1-b)\left(1-q^{-n-1}\right).
\end{align*} 
Note further that $u_{n+1}=0$ and $v_0=0$. 

Thus we have, with $u_k$, $v_k$ and $w_k$ as above:
\begin{align*}
\sum_{k=0}^{n+1}\left( F(n+1,k)-F(n,k)\right) &= 
\frac{(-1)c(1-a)(1-b)q^n\qrfac{c}{n}\qrfac{c/ab}{n}}{ab\qrfac{c/a}{n+1}\qrfac{c/b}{n+1}}\\
&\times 
\sum_{k=0}^{n+1} \frac{w_k}{w_0}\frac{u_0u_1\cdots u_{k-1}}{v_1v_2\cdots v_{k}}
\\
&= 0.
\end{align*}
Thus $$\sum_{k=0}^{n} F(n,k)$$
is a constant. To finish the proof, we verify that this sum is $1$ when $n=0$. 
\end{proof}

As we have seen, the proofs of the $q$-analogs presented above are analogous to those of the classical identities presented in \S\ref{sec:hyper}. In the next section, we prove the $q$-Dougall sum, a sum that encapsulates all the identities of this section. 

\section{Jackson's $\boldsymbol q$-analog of Dougall's Sum}\label{sec:dougall}

The algebra involved in the proof of the $q$-Pfaff-Saalsch\" utz sum is complicated enough to discourage us from trying to prove more complicated identities. Fortunately, there is an elementary identity that takes care of the algebra. In this section, we use a similar idea to prove a much more complicated identity,  namely the $q$-Dougall sum. The proof is again by the EZ method outlined in \S\ref{sec:WZ}. 

Ekhad and Zeilberger \cite{ekhzeil1} had earlier given a \lq\lq 21st century proof'' of Dou\-gall's sum.  Their work is important because the Dougall summation is a very general summation, and special cases include all the fundamental summation theorems in the theory of Hypergeometric Series. Likewise, if there is only one summation theorem that we can prove, then the $q$-Dougall summation is the one. This theorem encapsulates many of the other summation theorems---terminating and non-terminating---that comprise the theory of Basic Hypergeometric Series. Further, by suitably modifying the parameters and taking limits as $q\to 1$, one obtains all the main Hypergeometric sum identities too.  The reader may consult \S 2.7 of Gasper and Rahman~\cite{grhyp} to learn how the $q$-Dougall summation is specialized to obtain the key summations formulas in the theory of Basic Hypergeometric Series. 

At this time, its a good idea for the  reader to understand the notations used to write and describe $q$-series. While these notations are not strictly necessary to understand what follows, they are needed to understand how we obtain the elementary identities required for our work. These are special cases of identities from \cite{grhyp}. 

Basic hypergeometric series, or $q$-hypergeometric series, 
with $r$ numerator parameters $a_1$, $\ldots$, $a_r$ and
$s$ denominator parameters $b_1$, $\ldots$, $b_s$, 
and with base $q$ are defined as
\begin{align}\label{defhyp}
_r\phi_s\,\left[\begin{matrix}a_1,\dots,a_r\\
b_1,\dots,b_s\end{matrix};q,z\right]:=
\suml _{k=0} ^{\infty}\frac {(a_1;q)_k(a_2;q)_k \cdots (a_r;q)_k}
{(q;q)_k(b_1;q)_k \cdots (b_s;q)_k}\,
\left[(-1)^k q^{k\choose 2}\right]^{1+s-r}z^k,
\end{align}
with ${k\choose 2}=k(k-1)/2$, where $q\neq 0$ when $r>s+1$. 

For example, the $q$-Pfaff--Saalsch\" utz sum \eqref{qpfaff-saalchutz1} can be written as:
\begin{equation*}
_3\phi_2\left[
\begin{matrix}{a,  b, q^{-n}} \\
 {c, abq^{1-n}/c}
 \end{matrix}
\; ; q,\; q\right]
= \frac{\qrfac {c/ a}{\!n} \qrfac {c/b}{\!n}}
{\qrfac {c/ ab}{\!n} \qrfac {c}n }.
\end{equation*}

The $_3\phi_2$ series here is an example of a  {\em balanced} series. 
The term ``balanced'' refers to a condition
 which appears frequently in summation and
transformation formulas. 

An $_{r+1}\phi_r$ series is called $k$-balanced if in
 \eqref{defhyp},
$b_1\cdots b_r = a_1\cdots a_{r+1}q^k$ and $z=q$, and a $1$-balanced
series is called balanced.

Another condition that appears frequently in dealing with series is the very-well-poised condition.

An  $_{r+1}\phi_r$ 
series is well-poised if in \eqref{defhyp},
$qa_1=a_2b_1=\cdots = a_{r+1}b_r$. It is called very-well-poised if
it is well-poised and if $a_2=q{{\sqrt{a_1}}}$ and 
$a_3 = -q{{\sqrt{a_1}}}$. 

The objective of this section is to prove a summation theorem for a balanced, very-well-poised $_8\phi_7$ series found
by  Jackson
\cite[Equation~(2.6.2)]{grhyp}:
\begin{align}\label{class87gl}\allowdisplaybreaks\displaybreak[0]
_8\phi_7\,&\left[\begin{matrix}a,\,q\sqrt{a},-q\sqrt{a},b,c,d,
a^2q^{n+1}/bcd,q^{-n}\\
\sqrt{a},-\sqrt{a},aq/b,aq/c,aq/d,
bcdq^{-n}/a,aq^{n+1}\end{matrix};q,q\right]\cr
&=\frac {(aq;q)_n\,(aq/bc;q)_n\,(aq/bd;q)_n\,(aq/cd;q)_n}
{(aq/b;q)_n\,(aq/c;q)_n\,(aq/d;q)_n\,(aq/bcd;q)_n}.
\end{align}

First, we examine our proof of the $q$-Pfaff--Saalsch\" utz formula. Note that the tough part of the proof of \eqref{qpfaff-saalchutz1} is to show that:
\begin{align*}
&\left( \left( 1-cq^n\right) \left(1-abq^{-n+k}/c\right) \left(1-q^{-n-1}\right)
\left( cq^n/ab\right)
\right.
\\
&+ \left. \left( 1-cq^n/a\right)\left( 1-cq^n/b\right) \left(1-q^{-n+k-1}\right)\right)\\
&=\left(
\left(1-cq^n\right)\left( 1-c/abq \right) \left(1-q^{k}\right) \right.\\
&+ \left. \frac{cq^n}{ab}(1-a)(1-b) \left(1-q^{-n+k-1}\right) 
\right) 
\\
&=
\left(\frac{c}{ab}q^{n-k}\right)
\left[ \left(1-aq^k\right)\left(1-bq^k\right)\left( 1-q^{-n-1+k}\right)\right. \\
&-\left. \left(1-cq^{k-1}\right)\left(1-abq^{-n+k}/c\right)\left(1-q^k\right)\right].
\end{align*}

The first equality is required to simplify $F(n+1,k)-F(n,k)$, and the second to find an expression for $w_k=u_k-v_k.$ Let us dispense with the middle step, that we wrote for aesthetic reasons. We find that we have to prove that:

\begin{align}
 \left(1-aq^k\right)&\left(1-bq^k\right)\left( 1-q^{-n-1+k}\right) \nonumber \\
&- \left(1-cq^{k-1}\right)\left(1-abq^{-n+k}/c\right)\left(1-q^k\right)  \nonumber \\
=
 \left(abq^{-n+k}/c\right)&
\left[ \left( 1-cq^n\right) \left(1-abq^{-n+k}/c\right) \left(1-q^{-n-1}\right)
\left( cq^n/ab\right)
\right. \nonumber 
\\
&+ \left. \left( 1-cq^n/a\right)\left( 1-cq^n/b\right) \left(1-q^{-n+k-1}\right)\right].
\label{sears1}
\end{align}
The left hand side of \eqref{sears1} can be written as:

\begin{align*}
 \left(1-aq^k\right)&\left(1-bq^k\right)\left( 1-q^{-n-1+k}\right)\\
& \times
 \left[1 - 
 \frac{
\left(1-cq^{k-1}\right)\left(1-abq^{-n+k}/c\right)\left(1-q^k\right)}
 {\left(1-aq^k\right)\left(1-bq^k\right)\left( 1-q^{-n-1+k}\right)}
\right].
\end{align*}
The two terms in the brackets can be transformed by the $a\mapsto abq^{-n+k}/c$, $b\mapsto cq^{k-1}$, $c\mapsto q^k$, $d\mapsto q^{-n+k-1}$,  $e\mapsto aq^k$, and $f\mapsto bq^k$ case of the elementary relation:
\begin{align}
1-\frac{(1-a)(1-b)(1-c)}{(1-d)(1-e)(1-f)} &=
\frac{(1-e/a)(1-f/a)}{(1-e)(1-f)} a\nonumber \\
&\times \left[
1-\frac{(1-a)(1-d/b)(1-d/c)}{(1-d)(1-a/e)(1-a/f)}
\right], \label{searsn1}
\end{align}
where $abc=def$.
On using this special case of \eqref{searsn1}, we immediately obtain the RHS of \eqref{sears1}.

The elementary identity \eqref{searsn1} is the $n=1$ case of the well-known Sears' $_4\phi_3$ transformation formula \cite[eq.~(2.10.4)]{grhyp}, that transforms a terminating, balanced, 
$_4\phi_3$ series into a multiple of another such series.

Our proof of Jackson's formula \eqref{class87gl} relies upon the $n=1$ case  another well-known transformation formula, namely  the $_{10}\phi_9$ transformation formula \cite[eq. (2.9.1)]{grhyp}, that transforms a very-well-poised and balanced $_{10}\phi_9$ series into a multiple of another such series. We can take $n=1$ and $a\mapsto a/q$ in 
\cite[eq. (2.9.1)]{grhyp} to obtain the following elementary identity:
\begin{align}
&1-\frac{(1-b)(1-c)(1-d)(1-e)(1-f)(1-a^3/bcdef)}
{(1-a/b)(1-a/c)(1-a/d)(1-a/e)(1-a/f)(1-bcdef/a^2)}\nonumber \\
 &=
\frac{(1-a)(1-a/ef)(1-a^2/bcde)(1-a^2/bcdf)}
{(1-a/e)(1-a/f)(1-a^2/bcd)(1-a^2/bcdef)} \nonumber \\
&\times \left[
1-\frac{(1-a/bc)(1-a/bd)(1-a/cd)(1-e)(1-f)(1-a^3/bcdef)}
{(1-a/b)(1-a/c)(1-a/d)(1-a^2/bcde)(1-a^2/bcdf)(1-ef/a)}
\right]. \label{10p9n1}
\end{align}
For our proof of Jackson's sum, we need to iterate this elementary result. That is, we need to take the two terms in the bracket on the RHS of \eqref{10p9n1}, and make them fit the LHS, and once again apply \eqref{10p9n1}. 
For this we use the following case of \eqref{10p9n1}:
$a\mapsto a^2/bcd$, $b\mapsto e$, $c\mapsto f$, $d\mapsto a/bc$,  $e\mapsto a/cd$, and $f\mapsto a/bd$. 

In this manner, we obtain the elementary relation:
\begin{align}
&1-\frac{(1-b)(1-c)(1-d)(1-e)(1-f)(1-a^3/bcdef)}
{(1-a/b)(1-a/c)(1-a/d)(1-a/e)(1-a/f)(1-bcdef/a^2)}\nonumber \\
 &=
\frac{(1-a)(1-d)(1-a^2/bcde)(1-a^2/bcdf)(1-a^2/bdef)(1-a^2/cdef)}
{(1-a/b)(1-a/c)(1-a/e)(1-a/f)(1-a^2/bcdef)(1-a^3/bcd^2ef)} \nonumber \\
&\times \left[1-
\frac{(1-a/bd)(1-a/cd)(1-a/de)(1-a/df)(1-\frac{a^2}{bcdef})(1-\frac{a^3}{bcdef})}
{(1-1/d)(1-a/d)(1-\frac{a^2}{bcde})(1-\frac{a^2}{bcdf})(1-\frac{a^2}{bdef})(1-\frac{a^2}{cdef})}
\right]. \label{10p9n1-2}
\end{align}

A motive for obtaining \eqref{10p9n1-2} is that only one term in the numerator of the fraction on the LHS of \eqref{10p9n1-2} is repeated in the RHS. The same comment applies to the denominator. This is also a feature of \eqref{searsn1}.  

We are now ready to prove the $q$-Dougall summation. 

\begin{Example}[The $q$-Dougall sum (Jackson (1921))]Let $n$ be a non-negative integer and let $q$, $a$, $b$, $c$, and $d$ be such that the denominators in \eqref{qdougall2} are not zero. Then we have:
\begin{align}
\sum_{k=0}^n &
\frac{1-aq^{2k}}{1-a}
\frac{\qrfac{a}{k} \qrfac{b}{k}\qrfac{c}{k}\qrfac{d}{k}\qrfac{a^2q^{n+1}/bcd}{k}\qrfac{q^{-n}}{k}}
 {\qrfac{aq/b}{k}\qrfac{aq/c}{k}\qrfac{aq/d}{k}\qrfac{bcdq^{-n}/a}{k} \qrfac{aq^{n+1}}{k} \qrfac{q}{k}} q^k 
\nonumber\\
&=
\frac {(aq;q)_n\,(aq/bc;q)_n\,(aq/bd;q)_n\,(aq/cd;q)_n}
{(aq/b;q)_n\,(aq/c;q)_n\,(aq/d;q)_n\,(aq/bcd;q)_n}. \label{qdougall2}
\end{align}
\end{Example}
\begin{proof}
By dividing by the RHS, we form an equivalent identity of the form
\begin{equation*}
\sum_{k=0}^n F(n,k) = 1,
\end{equation*}
where $F(n,k)$ is defined as:
\begin{align*}
F(n,k)&= \frac{(aq/b;q)_n\,(aq/c;q)_n\,(aq/d;q)_n\,(aq/bcd;q)_n}
{(aq;q)_n\,(aq/bc;q)_n\,(aq/bd;q)_n\,(aq/cd;q)_n}\\
&\times \frac{1-aq^{2k}}{1-a}
\frac{\qrfac{a}{k} \qrfac{b}{k}\qrfac{c}{k}\qrfac{d}{k}}
 {\qrfac{aq/b}{k}\qrfac{aq/c}{k}\qrfac{aq/d}{k} \qrfac{q}{k}} q^k \\
&\times
 \frac{\qrfac{a^2q^{n+1}/bcd}{k}\qrfac{q^{-n}}{k}}
 {\qrfac{bcdq^{-n}/a}{k} \qrfac{aq^{n+1}}{k}} .
\end{align*}

We find that
\begin{align*}
F(n+1,k)-& F(n,k) = 
\frac{1-aq^{2k}}{1-a}
\frac{\qrfac{a}{k} \qrfac{b}{k}\qrfac{c}{k}\qrfac{d}{k}}
 {\qrfac{aq/b}{k}\qrfac{aq/c}{k}\qrfac{aq/d}{k}\qrfac{q}{k}} q^k 
\\
&\times 
\frac{(aq/b;q)_n\,(aq/c;q)_n\,(aq/d;q)_n\,(aq/bcd;q)_n}
{(aq;q)_{n+1}\,(aq/bc;q)_{n+1}\,(aq/bd;q)_{n+1}\,(aq/cd;q)_{n+1}}
\\
\times &\frac{\qrfac{q^{-n-1}}{k} \qrfac{a^2q^{n+1}/bcd}{k}}
{\left(1-q^{-n-1}\right)\left( 1-a^2q^{n+1}/bcd\right)}
\frac{(-1)}{\qrfac{aq^{n+2}}{k}\qrfac{bcdq^{-n}/a}{k}}
\\
&
\left\{\walker
\left(1-q^{-n-1}\right) \left( 1-aq^{n+1}/b\right) \left( 1-aq^{n+1}/c\right)\left( 1-aq^{n+1}/d\right) 
\right.\\
&\times \left(1-bcdq^{-n+k-1}/a\right) \left( 1-a^2q^{n+k+1}/bcd\right)\left( aq^{n+1}/bcd\right)
\\
&+ 
 \left(1-q^{-n+k-1}\right) \left( 1-aq^{n+k+1}\right)  \left( 1-aq^{n+1}/bc\right) \\
&\times \left.
\left( 1-aq^{n+1}/bd\right)\left( 1-aq^{n+1}/cd\right)\left( 1-a^2q^{n+1}/bcd\right)\walker
\right\}.
\end{align*}

Natural candidates for $u_k$ and $v_k$ are as follows.  
\begin{align*}
u_k&=\left(1-aq^k\right)\left(1-bq^k\right)\left(1-cq^k\right)\left(1-dq^k\right)\\
&\times \left(1-a^2q^{n+k+1}/{bcd}\right)\left( 1-q^{-n-1+k}\right)\\
\text{and }&\\
 v_k & = \left(1-aq^k/b\right)\left(1-aq^k/c\right)\left(1-aq^k/d\right)\\
 &\times
 \left(1-bcdq^{-n+k-1}/a\right)
\left(1-aq^{n+k+1}\right)\left( 1-q^{k}\right).
\end{align*}

Note  that $u_{n+1}=0$ and $v_0=0$, and 
$$w_0=u_0-v_0=
\left(1-a\right)\left(1-b\right)\left(1-c\right)\left(1-d\right)
 \left(1-\frac{a^2}{bcd}q^{n+1}\right)\left( 1-q^{-n-1}\right).$$

We now use the following case of \eqref{10p9n1-2}:
$a\mapsto aq^{2k}$, $b\mapsto aq^k/b$, $c\mapsto aq^k/c$, 
$d\mapsto aq^{n+k+1}$,  $e\mapsto aq^k/d$, and $f\mapsto q^k$
to rewrite $$\left(1-\frac{v_k}{u_k}\right).$$
In this manner, we obtain:

\begin{align*}
w_k&=u_k- v_k = \frac{\left(1-aq^{2k}\right)}{\left(1-a^2q^{2n+2}/bcd\right)}
\left(\frac{q^k}{aq^{n+1}}\right)
\\
&\times
\left\{\walker
\left(1-q^{-n-1}\right) \left( 1-aq^{n+1}/b\right) \left( 1-aq^{n+1}/c\right)\left( 1-aq^{n+1}/d\right) 
\right.\\
&\times \left(1-bcdq^{-n+k-1}/a\right) \left( 1-a^2q^{n+k+1}/bcd\right)\left( aq^{n+1}/bcd\right)
\\
&+ 
 \left(1-q^{-n+k-1}\right) \left( 1-aq^{n+k+1}\right)  \left( 1-aq^{n+1}/bc\right) \\
&\times \left.
\left( 1-aq^{n+1}/bd\right)\left( 1-aq^{n+1}/cd\right)\left( 1-a^2q^{n+1}/bcd\right)\walker
\right\}.
\end{align*}
The expression in the braces is the same as the expression in braces on the RHS of the expression for $F(n+1,k)-F(n,k)$. 

Thus we have, with $u_k$, $v_k$ and $w_k$ as above:
\begin{align*}
\sum_{k=0}^{n+1}&\left( F(n+1,k)-F(n,k)\right) \\
&= 
\frac{(aq/b;q)_n\,(aq/c;q)_n\,(aq/d;q)_n\,(aq/bcd;q)_n}
{(aq;q)_{n+1}\,(aq/bc;q)_{n+1}\,(aq/bd;q)_{n+1}\,(aq/cd;q)_{n+1}}\\
&\times 
\left[
\left(-aq^{n+1}\right)(1-b)(1-c)(1-d)\left( 1-a^2q^{2n+2}/bcd\right)\right]\\
&\times 
\sum_{k=0}^{n+1} \frac{w_k}{w_0}\frac{u_0u_1\cdots u_{k-1}}{v_1v_2\cdots v_{k}}
\\
&= 0.
\end{align*}
Thus $$\sum_{k=0}^{n} F(n,k)$$ is a constant. To finish the proof, we verify that this sum is $1$ when $n=0$. 
\end{proof}

For the sake of completeness, we note that if we take the limit as $d\to \infty$ in Jackson's sum \eqref{class87gl}, we obtain Rogers' sum for a terminating, very-well-poised $_6\phi_5$ sum
\cite[Equation~(2.4.2)]{grhyp}:
\begin{align}\label{class65gl}\allowdisplaybreaks\displaybreak[0]
_6\phi_5\,&\left[\begin{matrix}a,\,q\sqrt{a},-q\sqrt{a},b,c,q^{-n}\\
\sqrt{a},-\sqrt{a},aq/b,aq/c,
aq^{n+1}\end{matrix};q,\frac{aq^{n+1}}{bc}\right]\cr
&=\frac {(aq;q)_n\,(aq/bc;q)_n}
{(aq/b;q)_n\,(aq/c;q)_n}.
\end{align}
Indeed, one can suitably specialize the proof of the $q$-Dougall sum given above by taking limits as $d\to \infty$, make minor modifications in the choice of $u_k$ and $v_k$, and obtain a proof of \eqref{class65gl} by using Euler's Telescoping Lemma. 

Further, one may suitably specialize the parameters and take the limit $q\to 1$ to obtain Ekhad and Zeilberger's proof of Dougall's sum.  However, our proof is not quite a 21st century proof, because it has been found without using a computer. It uses the 
$_{10}\phi_9$ transformation formula which Bailey found in 1929, besides Euler's Telescoping Lemma and the WZ trick.  Nevertheless, conceptually, it is really a WZ proof.

%

\section{Generalized Hypergeometric Series} \label{sec:generalized}

If we have an elementary relation of the form
$$U-V=W$$
then we can get an identity 
using the Telescoping Lemma \eqref{telescoping-lemma}. For example, consider the elementary identity:
\begin{equation}\label{qchv-elem1}
 (1-b)a-(1-a)b=a-b.
 \end{equation}
This implies that for sequences $a_k$ and $b_k$, we have:
\begin{equation}\label{qchv-elem2}
(1-b_k)a_k-(1-a_k)b_k=a_k-b_k.
\end{equation}
Now one can appeal to the Telescoping Lemma \eqref{telescoping-lemma}, with 
$u_k = (1-b_k)a_k$, and $v_k= (1-a_k)b_k$ (so that $w_k=a_k-b_k$) and obtain:
\begin{align}\label{cv1}
\sum_{k=0}^n \frac{a_k-b_k}{a_0-b_0} &
\frac{\prodl_{j=0}^{k-1} (1-b_j)a_j}{\prodl_{j=1}^{k} (1-a_j)b_j}\nonumber
\\
&=\frac{(1-b_0)a_0}{a_0-b_0}\left(
\prod_{j=1}^n \frac{(1-b_j)a_j}{(1-a_j)b_j}
-
\frac{(1-a_0)b_0}{(1-b_0)a_0}\right).
\end{align}
We assume that the sequences $a_k$ and $b_k$ are such that none of the denominators in \eqref{cv1} are $0$. 

\begin{Remark} Identity \eqref{cv1} is an extension of Ramanujan's identity \eqref{ram1}. To recover \eqref{ram1} from \eqref{cv1}, set $a_k\mapsto -a_{k+1}/x$, and take the limits $b_k\to \infty$, for $k=0, 1, 2, \dots, n.$  
\end{Remark}

This kind of series may be called {\em Generalized Hypergeometric Series} because the sequences can be specialized to find hypergeometric or $q$-hyper\-geometric series.  

Observe that when $n=1$, the $q$-Chu--Vandermonde identity \eqref{qchu-vandermonde1} reduces to:
$$ 1 + \frac{(1-a)(1-q^{-1})}{(1-b)(1-q)}\frac{bq}{a} =\frac{1-b/a}{1-b}.$$
This simplifies to \eqref{qchv-elem1}. 
Our idea in this section is to take $n=1$ in $q$-series identities, and use \eqref{telescoping-lemma} to write down corresponding identities for generalized hypergeometric series. 

An alternative---and automated---approach of finding and proving identities with sequences as parameters is given by Kauers and Schneider \cite{kauers-schneider}. These authors mention Euler's identity and Ramanujan's identity, among other examples of such series. 

Ramanujan's sum appeared in Ap{\' e}ry's proof of the irrationality of $\zeta(3)$, see van der Poorten \cite{vanderpoorten}. Apart from Ramanujan, such identities have also been given by Gould and Hsu, Carlitz, Krattenthaler, Chu, and by Macdonald (see \cite{gb-milne} and the references cited therein).  
 
We now find a generalized hypergeometric summation that follows by combining the $n=1$ case of the $q$-Pfaff--Saalsch\" utz summation with Euler's Telescoping Lemma.

\begin{Example}[Macdonald] Let $a_k$, $b_k$ and $c_k$ be sequences, such that none of the denominators in \eqref{ps1} vanish. Then
\begin{align}\label{ps1}
\sum_{k=0}^n &\frac{(a_k-b_k)(a_k-c_k)}{(a_0-b_0)(a_0-c_0)} 
\frac{\prodl_{j=0}^{k-1} (1-b_j)(1-c_j)a_j}{\prodl_{j=1}^{k} (1-a_j)(a_j-b_jc_j)}\nonumber
\\
&=\frac{(1-b_0)(1-c_0)a_0}{(a_0-b_0)(a_0-c_0)}\nonumber
\\
\times &
\left(
\prod_{j=1}^n \frac{(1-b_j)(1-c_j)a_j}{(1-a_j)(a_j-b_jc_j)} 
-
\frac{(1-a_0)(a_0-b_0c_0)}{(1-b_0)(1-c_0)a_0}\right).
\end{align}
\end{Example}
\begin{Remark}
If we set $c_k\mapsto 1/a_k$, $a_k\mapsto b_k/b_0$, $b_k\mapsto c_kb_k/b_0$, we obtain  Macdonald's identity in the form presented in~\cite[eq.~1.13]{gb-milne}.
\end{Remark}
\begin{Remark}
If we take the limits $c_k\to\infty$ in \eqref{ps1}, we obtain \eqref{cv1}. 
\end{Remark}
\begin{proof}
We use the $n=1$ case of the  $q$-Pfaff--Saalsch\" utz  identity to find an elementary identity. Interchange $c$ and $a$ and set $n=1$ in \eqref{qpfaff-saalchutz1} to obtain: 
$$ 1 + \frac{(1-c)(1-b)(1-q^{-1})}{(1-a)(1-bc/a)(1-q)}q =\frac{(1-a/c)(1-a/b)}{(1-a)(1-a/bc)},$$
or
$$ (1-b)(1-c)a - (1-a)(a-bc)=(a-b)(a-c).$$
Replacing $a$, $b$ and $c$ by sequences $a_k$, $b_k$ and $c_k$, we obtain:
$$ (1-b_k)(1-c_k)a_k - (1-a_k)(a_k-b_kc_k)=(a_k-b_k)(a_k-c_k).$$

Now take $u_k = (1-b_k)(1-c_k)a_k,$ $v_k= (1-a_k)(a_k-b_kc_k),$ 
and  $w_k=(a_k-b_k)(a_k-c_k)$ in \eqref{telescoping-lemma} to obtain \eqref{ps1}. 
\end{proof}

\begin{Remark}
If we set $n=1$ in the very-well-poised sum $_6\phi_5$ sum \eqref{class65gl} we get the same elementary identity. 
\end{Remark}
 

\newcommand{\douguj}{(1-b_j)(1-c_j)(1-d_j)(a_j^2-b_jc_jd_j)a_j}

\newcommand{\douguz}{ (1-b_0)(1-c_0)(1-d_0)(a_0^2-b_0c_0d_0)a_0}
\newcommand{\douguk}{(1-b_k)(1-c_k)(1-d_k)(a_k^2-b_kc_kd_k)a_k}

\newcommand{\dougvj}{(1-a_j)(a_j-b_jc_j)(a_j-b_jd_j)(a_j-c_jd_j)}
\newcommand{\dougvz}{(1-a_0)(a_0-b_0c_0)(a_0-b_0d_0)(a_0-c_0d_0)}
\newcommand{\dougvk} {(1-a_k)(a_k-b_kc_k)(a_k-b_kd_k)(a_k-c_kd_k)}

\newcommand{\dougwj}{ (a_j-b_j)(a_j-c_j)(a_j-d_j)(a_j-b_jc_jd_j)}
\newcommand{\dougwz}{(a_0-b_0)(a_0-c_0)(a_0-d_0)(a_0-b_0c_0d_0)}
\newcommand{\dougwk}{(a_k-b_k)(a_k-c_k)(a_k-d_k)(a_k-b_kc_kd_k)}

Next, we obtain a generalized hypergeometric series summation by using Jackson's $q$-Dougall summation. 

\begin{Example}[Macdonald] Let $a_k$, $b_k$, $c_k$ and $d_k$ be sequences such that none of the denominators in \eqref{doug1} vanish. Then
\begin{align}\label{doug1}
\sum_{k=0}^n &\frac{ \dougwk}
{\dougwz} \nonumber \\
&\times \frac{\prodl_{j=0}^{k-1} \left[ \douguj \right]}
{\prodl_{j=1}^{k} \left[  \dougvj   \right]}   \nonumber
\\
&=\frac{ \douguz }
{\dougwz } \nonumber \\
&\times \left(  
\prod_{j=1}^n\left[ \frac{ \douguj}
{\dougvj } \right]\right. \nonumber\\
&-
\left. \walker \frac{ \dougvz}
{ \douguz}\right).
\end{align}
\end{Example}
\begin{Remark}
Set $a_k\mapsto b_ke$, $b_k\mapsto1/a_k$, $c_k\mapsto b_kc_k$, $d_k\mapsto a_k/d_k$ in 
identity \eqref{doug1} to obtain Macdonald's identity in the form  given in \cite[eq. (2.28)]{gb-milne}.
\end{Remark}
\begin{Remark}
If we set $d_k=0$ in \eqref{doug1}, we obtain \eqref{ps1}. 
\end{Remark}

\begin{proof}

Take $n=1$ and $a\mapsto a/q$ in the  $q$-Dougall sum \eqref{qdougall2}, to obtain the following elementary identity:
 \begin{align}
 (1-b)&(1-c)(1-d)(a^2-bcd)a\nonumber\\
&-
 (1-a)(a-bc)(a-bd)(a-cd)\nonumber\\
& =
 (a-b)(a-c)(a-d)(a-bcd)\label{doug-n1}
\end{align} 
Thus we have, for sequences $a_k$, $b_k$, $c_k$ and $d_k$:
 \begin{align*}
 (1-b_k)&(1-c_k)(1-d_k)(a_k^2-b_kc_kd_k)a_k\\
&-
 (1-a_k)(a_k-b_kc_k)(a_k-b_kd_k)(a_k-c_kd_k)\\
& =
 (a_k-b_k)(a_k-c_k)(a_k-d_k)(a_k-b_kc_kd_k).
\end{align*} 

Now using \eqref{telescoping-lemma}, we immediately obtain \eqref{doug1}
\end{proof}

We conclude this section with a few remarks on the symmetry properties of the Generalized Hypergeometric Series presented in this section. 

Let us represent the equation $U-V=W$ by $(U, V, W)$.   Given an equation of the form $U-V=W$, we can list 6 possible permutations of this equation. They are: $(U, V, W)$, $(U, W, V)$,  $(W, -V, U)$,  $(V, U, -W)$, $(W, U, -V)$, and $(V, -W, U)$. We can set $(u_k, v_k, w_k)$ to be equal to any of these  permutations, and potentially derive a different identity. 

However, it so happens, that for the  examples considered in this section, all the identities obtained by these permutations are equivalent to each other.

For example,  \eqref{qchv-elem2} is of the form $(U, V, W)$. If $(u_k, v_k, w_k)$ corresponds to $(U, V, W)$, then we  obtain \eqref{cv1}. Instead,  take  $(u_k, v_k, w_k)$ to correspond to $(U,W, V)$, that is, take  
 $u_k = (1-b_k)a_k,$ $v_k= a_k-b_k,$ and  $w_k=(1-a_k)b_k$ to obtain
\begin{align}\label{cv2}
\sum_{k=0}^n \frac{(1-a_k)b_k}{(1-a_0)b_0} &
\frac{\prodl_{j=0}^{k-1} (1-b_j)a_j}{\prodl_{j=1}^{k} (a_j-b_j)}\nonumber
\\
&=\frac{(1-b_0)a_0}{(1-a_0)b_0}\left(
\prod_{j=1}^n \frac{(1-b_j)a_j}{(a_j-b_j)}
-
\frac{a_0-b_0}{(1-b_0)a_0}\right).
\end{align}
Identity \eqref{cv2} can be obtained from \eqref{cv1} by re-labeling parameters. 
Set $a_k\mapsto a_k/b_k$  and $b_k\mapsto 1/b_k$ in \eqref{cv1}  to obtain \eqref{cv2}. 

By examining $(u_k, v_k, w_k)$ as they appear in the examples of this section, it seems remarkable that identities obtained by permuting the three are all equivalent. 
However, things become clear once we examine the  remarkable symmetry exhibited by the following 
elementary identity~\cite[eq. (11.1.1)]{grhyp}
\begin{align}
(1-x\lambda)&(1-x/\lambda)(1-\mu\nu)(1-\mu/\nu)-
(1-x\nu)(1-x/\nu)(1-\lambda\mu)(1-\mu/\lambda)\nonumber \\
&=\frac{\mu}{\lambda}(1-x\mu)(1-x/\mu)(1-\lambda\nu)(1-\lambda/\nu).\label{doug-sym}
\end{align} 
This beautiful identity follows from the $n=1$ case of  the $q$-Dougall sum.  Set $a\mapsto x\nu$, $b\mapsto x\lambda$, $c\mapsto x/\lambda$ and $d\mapsto \mu\nu$ in \eqref{doug-n1} to obtain \eqref{doug-sym}. Thus the symmetry of this elementary identity is responsible for the many symmetries of Macdonald's identity.

It is interesting to note the  important role played by elementary identities (obtained by taking $n=1$ in summation and transformation formulas) in \S\ref{sec:dougall} and \S\ref{sec:generalized}. These elementary identities appear in Andrews \cite{and-pfaff2} and Guo and Zeng \cite{guo-zeng} too. 


In the development of Generalized Hypergeometric Series, 
a key step was Krattenthaler's \cite{krat2} matrix inverse that generalized Andrews' matrix formulation of the Bailey Transform and the extension by Agarwal, Andrews and Bressoud \cite{aab}. 
Krattenthaler's Matrix Inverse was extended by Chu \cite{chu} using telescoping.  Macdonald's proof of Chu's results used the Telescoping Lemma, and  generalized Chu's results to the Generalized Hypergeometric Series presented in \cite{gb-milne} and in this section. See also
Kauers and Schneider \cite{kauers-schneider} for more examples of such series. 

Ramanujan's identity makes an interesting appearance in van der Poorten \cite{vanderpoorten}.  We refer the reader to Gasper and Rahman~\cite[\S 11.6]{grhyp} for Warnaar's \cite{warnaar1} extension of Macdonald's identity to theta hypergeometric series. The proof of Warnaar's identity is on the lines of the proof of \eqref{doug1}, and uses Euler's Telescoping Lemma. See also Spiridonov \cite{spiridonov}.

\section{Three-term Recurrence Relations}\label{sec:three-term}

In view of Example~\ref{lucas}, it is clear that Euler's Telescoping Lemma applies whenever we have a three-term recurrence relation such as that of the Fibonacci Numbers. 
For example, consider the beautiful extension of the Rogers--Ramanujan identities given by Garrett, Ismail and Stanton \cite{gis}: 
\begin{equation}\label{gis-rr}
\sum_{k=0}^{\infty} \frac{q^{k^2+km}}{\qrfac{q}{k}}
= \frac{(-1)^mq^{-{m\choose 2}}}{\pqrfac{q}{\infty}{q^5}\pqrfac{q^4}{\infty}{q^5}} d_m
-
\frac{(-1)^mq^{-{m\choose 2}}}{\pqrfac{q^2}{\infty}{q^5}\pqrfac{q^3}{\infty}{q^5}} e_m,
\end{equation}
where $m\geq 0$; and $d_m$ and $e_m$ satisfy the recurrence relations:
$$x_{n+2}=x_{n+1}+q^nx_n,$$
with the initial conditions:
$d_0=1,$  $d_1=1$ and $e_0=0,$ $e_1=1$.
The polynomials $d_m$ and $e_m$ are $q$-analogs of the Fibonacci numbers, and appear in Schur's original proof of the Rogers--Ramanujan identities, see \cite{and2}.

In this section we note a few applications of the Telescoping Lemma to such recurrences. In the process, we find that Euler's Telescoping lemma gives easy  alternate proofs of many identities found by Andrews \cite{and2, and3, and4}, Garrett \cite{garrett}, Briggs, Little and Sellers \cite{bls}, 
Cigler \cite{cigler-fib}, Goyt and Sagan \cite{goyt-sagan}, Goyt and Mathisen \cite{goyt-mathisen}, and Ismail \cite{ismail-fib}. 

In addition, we are able to find several identities that are extensions of Example~\ref{lucas} to the sequences considered by these authors. 

We begin our study of three-term recurrences by finding a useful solution of a linear recurrence relation.
\begin{Proposition}[Linear Recurrence Relation]
Let  $b_0, b_1, b_2,\dots$ and $c_0,$ $c_1,$ $c_2,$ $\dots$ be sequences.
Consider a sequence $x_0,$ $x_1,$ $x_2,$ $x_3 \dots $ that satisfies
\begin{equation}\label{linear-rec}
x_{n+1} =b_nx_n +c_n \text{ $(n\geq 0)$},
\end{equation}
and the value of $x_0$ is given.  Then we have:
\begin{equation}\label{linear-rec-soln}
x_{n+1}= x_0b_0b_1b_2\cdots b_n 
+ b_1b_2\cdots b_n
\sum_{k=0}^n \frac{c_k}{b_1b_2\cdots b_{k}}.
\end{equation}
\end{Proposition}
\begin{Remark}
Recursion \eqref{linear-rec} is a prototype of a {\em linear recurrence relation}, see Wilf \cite{wilf-alg}.
\end{Remark}

\begin{proof}
 We will use Euler's Telescoping Lemma to solve this recursion. 

Set $u_k=x_{k+1}$ and $v_k=b_kx_k$ in \eqref{telescoping-lemma}. Then $w_k=c_k$ and $w_0=c_0$. We obtain:
$$
\sum_{k=0}^n \frac{c_k}{c_0}\frac{x_1\cdots x_{k}}{b_1b_2\cdots b_{k}x_1\cdots x_{k}}
= \frac{x_1}{c_0}\left( \frac{x_2x_3\cdots x_{n+1}}{b_1b_2\cdots b_{n}x_1x_2\cdots x_{n}}
 - \frac{b_0x_0}{x_1}\right).
$$
This immediately gives us \eqref{linear-rec-soln}.
\end{proof}
\begin{Example} Let $x_n$ be a sequence defined as:
\begin{equation*}
x_{n+1} =nx_n +(-1)^n \text{ $(n\geq 0)$},
\end{equation*}
with $x_0 =0$. 
Then from \eqref{linear-rec-soln}, we obtain
$$\frac{x_{n+1}}{n!} = \sum_{k=0}^n \frac{(-1)^k}{k!} = 1-\frac{1}{1!}+\frac{1}{2!}-\frac{1}{3!}
+\cdots +(-1)^n \frac{1}{n!}.
$$
\end{Example}

The reader may recognize that $x_{n+1}=d_n,$ where $d_n$ is the number of derangements of a partition on $n$ letters.  
The derangement numbers also satisfy a three term recurrence relation (see, for example, \cite{cameron}):
$$d_{n+2}=(n+1)d_{n+1}+(n+1)d_n,$$
with $d_0=1$ and $d_1=0$. 
(Coincidentally, both these recursions for the derangement numbers are due to Euler, see Hopkins and Wilson \cite{hw}.)

We consider below other interesting sequences (of numbers and polynomials) that appear in combinatorial contexts. 

First, we generalize the identities in Example~\ref{lucas}. 

\begin{Theorem}\label{lucas-gen}
Let  $a_n$ and $b_n$ be sequences, with $a_n\neq 0$, $b_n\neq 0$, for all $n$. 
Consider a sequence $x_n$ that satisfies (for $n\geq 0$):
$$x_{n+2}=a_nx_{n+1}+b_nx_{n}.$$
Then the following identities hold for $n=0,1,2,\dots$, provided the denominators are not $0$. 
\begin{align}
\sum_{k=1}^n \frac{b_k}{a_1a_2\cdots a_k}\frac{x_k}{x_2} 
&=\frac{1}{a_1a_2\cdots a_n}\frac{ x_{n+2}}{x_2}-1.
\label{lucas-gen1}\\
\sum_{k=1}^n \frac{a_{2k-1}}{b_1b_3\cdots b_{2k-1}}\frac{x_{2k}}{x_1}
 &= \frac{1}{b_1b_3\cdots b_{2n-1}}\frac{x_{2n+1}}{x_1}-1.\label{lucas-gen-even}\\
\sum_{k=1}^n \frac{a_{2k}}{b_2b_4\cdots b_{2k}}\frac{x_{2k+1}}{x_2}
 &= \frac{1}{b_2b_4\cdots b_{2n}}\frac{x_{2n+2}}{x_2}-1.\label{lucas-gen-odd}\\
\sum_{k=1}^n \frac{a_{k}}{b_1b_2\cdots b_{k}}\frac{x_{k+1}^2}{x_1x_2}
 &= \frac{1}{b_1b_2\cdots b_{n}}\frac{x_{n+1}x_{n+2}}{x_1x_2}-1.
 \label{lucas-gen-square}\\
\sum_{k=1}^n (-1)^k \frac{a_1a_2\cdots a_{k-1}}{b_1b_2\cdots b_{k}}\frac{x_{k+2}}{x_1}
 &= (-1)^n\frac{a_1a_2\cdots a_{n}}{b_1b_2\cdots b_{n}}\frac{x_{n+1}}{x_1}-1.
 \label{lucas-gen-alternating}\\
\sum_{k=1}^n \frac{b_{k-1}b_k}{a_{k-1}}
\prod_{j=1}^k \left[ \frac{a_{j-1}}{a_{j-1}a_j+b_j}\right]
 \frac{x_{k-1}}{x_2}
& =  1-\prod_{j=1}^n \left[ \frac{a_{j-1}}{a_{j-1}a_j+b_j}\right]
  \frac{x_{n+2}}{x_2}.
 \label{lucas-gen-div2n}
 \end{align}
\end{Theorem}
\begin{Remark} There are many other Fibonacci Identities that follow from telescoping (see Vajda~\cite{vajda}), and can be extended as in Theorem~\ref{lucas-gen}. We do not undertake a detailed study here, we aim only  to illustrate the usefulness of Euler's Telescoping Lemma. 
\end{Remark}

\begin{proof}
The proofs are almost identical to the proofs of the corresponding identities in 
Example~\ref{lucas}. All these identities are special cases of \eqref{eulerfinite2}. 

To prove the first identity, we set $u_k=x_{k+2},$ and  $v_k=a_kx_{k+1}.$ Note that $w_k=b_kx_k.$ Substituting in \eqref{eulerfinite2}, we immediately obtain \eqref{lucas-gen1}. 

Next, set $u_k=x_{2k+1},$ and  $v_k=b_{2k-1}x_{2k-1}.$ Note that $w_k=a_{2k-1}x_{2k}.$ Substituting in \eqref{eulerfinite2}, we \eqref{lucas-gen-even}.

Similarly, set $u_k=x_{2k+2},$ and  $v_k=b_{2k}x_{2k}.$ Note that $w_k=a_{2k}x_{2k+1}.$ Substituting in \eqref{eulerfinite2}, we obtain
 \eqref{lucas-gen-odd}. 

Next,  set $u_k=x_{k+1}x_{k+2},$ and  $v_k=b_kx_{k}x_{k+1}.$ Note that $w_k=a_kx_{k+1}^2.$ Substituting in \eqref{eulerfinite2}, we obtain
\eqref{lucas-gen-square}.  
 
 Next,  set $u_k=a_kx_{k+1},$ and  $v_k=-b_kx_{k}.$ Note that $w_k=x_{k+2}.$ Substituting in \eqref{eulerfinite2}, we obtain
\eqref{lucas-gen-alternating}.  
 
 Finally,  set $u_k=x_{k+2},$ and  $v_k=(a_k+b_k/a_{k-1})x_{k+1}.$ It is easy to see that $w_k=-b_{k-1}b_kx_{k-1}/a_{k-1}.$ Substituting in \eqref{eulerfinite2}, we obtain
\eqref{lucas-gen-div2n}.  
 \end{proof}

As our first example, we find identities for the derangement numbers that are analogous to those for the Fibonacci numbers.  Note that since $d_1=0$, and some of our identities require division by $x_1$, we take a shifted sequence to ensure that we do not divide by $0$. So we define $D_n=d_{n+1}$ in our next example. 

 \begin{Example}[Derangement Number Identities]\label{derangement}
Consider the shifted Derangement Numbers defined as:
$D_0=0, D_1=1;$ and for $n\geq 0$,
$$D_{n+2}=(n+2)D_{n+1}+(n+2)D_{n}.$$
Then the following identities hold for $n=0,1,2,\dots$:
\begin{align}
\sum_{k=1}^n \frac{D_k}{(k+1)!} &= \frac{D_{n+2}}{(n+2)!}-1.\label{derangement1}\\
\sum_{k=1}^n \frac{D_{2k}}{1\cdot 3 \cdots (2k-1)} 
&= \frac{D_{2n+1}}{1\cdot 3 \cdots (2n+1)} -1.\label{derangement-even}\\
\sum_{k=1}^n \frac{D_{2k+1}}{2\cdot 4 \cdots (2k)} 
 &=\frac{D_{2n+2}}{2\cdot 4 \cdots (2n+2)}-1.\label{derangement-odd}\\
\sum_{k=1}^n \frac{D_{k+1}^2}{(k+1)!} 
 &= \frac{D_{n+1}D_{n+2}}{(n+2)!}-1.\label{derangement-square}\\
\sum_{k=1}^n (-1)^{k}\frac{D_{k+2}}{k+2} 
&= (-1)^{n}D_{n+1}-1.\label{derangement-alternating}\\
\sum_{k=1}^n \frac{2D_{k-1}}{(k+2)(k+1)!} 
&= 1- \frac{2D_{n+2}}{(n+2)(n+2)!}.\label{derangement-div-2n}
\end{align}
\end{Example}
\begin{Remark} We are unaware of the provenance of these identities. Derangement numbers are usually not considered to be analogous to the Fibonacci numbers in any way.  One can replace $D_k$ by $d_{k+1}$ to obtain identities for derangement numbers. 
\end{Remark}
\begin{proof}
These identities follow immediately from Theorem~\ref{lucas-gen} by taking
$a_k=k+2=b_k$ and  
replacing $x_n$ by 
$D_n$. Note also the initial conditions and the recurrence relation implies that $D_2=2$. 
\end{proof}

Another elementary sequence very similar to the Fibonacci sequence is the sequence of Pell numbers. The corresponding identities are given in the next example. 

\begin{Example}[Pell Number Identities]\label{pell}
Consider the Pell Numbers defined as:
$P_0=0, P_1=1;$ and for $n\geq 0$,
$$P_{n+2}=2P_{n+1}+P_{n}.$$
Then the following identities hold for $n=0,1,2,\dots$: 
\begin{align}
\sum_{k=1}^n \frac{P_k}{2^{k+1}} &= \frac{P_{n+2}}{2^{n+1}}-1.\label{pell1}\\
\sum_{k=1}^n 2P_{2k} &= P_{2n+1}-1.\label{pelleven}\\
\sum_{k=1}^n 2P_{2k-1} &= P_{2n}.\label{pellodd}\\
\sum_{k=1}^n 2P_{k}^2 &= P_{n}P_{n+1}.\label{pell-square}\\
\sum_{k=0}^n (-1)^{k}2^{k-1}P_{k+2} &= (-1)^{n}2^nP_{n+1}.\label{pell-alternating}\\
\sum_{k=1}^n \left(\frac{2}{5}\right)^{\! k}\frac{P_{k-1}}{2^2} &= 1- 
\left(\frac{2}{ 5}\right)^{\! n}\frac{P_{n+2}}{2}.\label{pell-div-2n}
\end{align}
\end{Example}
\begin{Remark} The first three identities are special cases of Pell Polynomials considered by Horadam and Mahon \cite{horadam-mahon}. Identity \eqref{pell-square} appears in Bicknell
\cite{bicknell}. 
\end{Remark}
\begin{proof}
These identities follow immediately from Theorem~\ref{lucas-gen} by taking
$a_k=2$ and $b_k=1$ and  
replacing $x_n$ by 
$P_n$. Note also the initial conditions and the recurrence relation implies that $P_2=2$. In the proof of \eqref{pellodd}, \eqref{pell-square} and \eqref{pell-alternating}, we bring $1$ to the LHS and may perform calculations similar to the proofs of the corresponding Fibonacci identities. 
\end{proof}

Our next example is of $q$-analogs of the Fibonacci Numbers that appeared in the work of Schur.
\begin{Example}[Schur's $q$-Fibonacci Numbers]\label{schur-q-fib} Consider the $q$-Fibonacci numbers defined as:
$F_0^{(a)}(q)=0$, $F_1^{(a)}(q)=1$, and
$$F_{n+2}^{(a)}(q) = F_{n+1}^{(a)}(q) + q^{n+a} F_n^{(a)}(q).$$
 Then we have, for $n=0,1,2,\dots$:
\begin{align}
\sum_{k=1}^n q^{k+a}F_k^{(a)}(q) &= F_{n+2}^{(a)}(q)-1.\label{and-lucas1}\\
\sum_{k=1}^n q^{-k^2-ka}F_{2k}^{(a)}(q) &= 
q^{-n^2-na}F_{2n+1}^{(a)}(q)-1.\label{garrett-lucaseven}\\
\sum_{k=1}^n q^{-(k-1)(k+a)}F_{2k-1}^{(a)}(q) &= 
q^{-(n-1)(n+a)}F_{2n}^{(a)}(q).\label{garrett-lucasodd}\\
\sum_{k=1}^n q^{-{{k\choose 2}}-(k-1)a}\left(F_{k}^{(a)}(q)\right)^2 &= 
q^{-{{n\choose 2}-(n-1)a}}F_{n}^{(a)}(q)F_{n+1}^{(a)}(q).\label{garrett-lucas-square}\\
\sum_{k=1}^n (-1)^{k-1} q^{-{{k\choose 2}}-(k-1)a}F_{k+1}^{(a)}(q) &= 
(-1)^{n-1} q^{-{{n\choose 2}-(n-1)a}}F_{n}^{(a)}(q).\label{qfib-lucas-alternating}\\
\sum_{k=1}^n q^{{2k-1}+2a}\frac{F_{k-1}^{(a)}(q)}{\qrfac{-q^{a+1}}{k}}&= 
1- \frac{F_{n+2}^{(a)}(q)}{\qrfac{-q^{a+1}}{n}} \label{qfib-lucas-2n}
\end{align}
\end{Example}
\begin{Remark} The notation we use is from Garrett \cite{garrett}. 
Note that $F_n^{(0)}(q)=e_n$ and $F_n^{(1)}(q)=d_{n+1}$, where $e_n$ and $d_n$ are the sequences that appear in \eqref{gis-rr}. Further, in the notation of Andrews \cite{and4}, we have
$F_n^{(0)}(q)=\mathcal{S}_n(q)$ and $F_n^{(1)}(q)=\mathcal{T}_{n}(q)$.
The $a=0$ case of \eqref{and-lucas1} is due to Andrews \cite{and2, and3}, see also Carlitz 
\cite{carlitz1, carlitz2} and Andrews \cite[Theorem 2]{and4}. The $a=0$ case of the identities \eqref{garrett-lucaseven}, \eqref{garrett-lucasodd}, \eqref{garrett-lucas-square} are due to Garrett \cite{garrett}. The last two identities appear to be new, as do the first four identities for $a\neq 0$.  When $a\neq 0$, then the numbers are called {\em shifted} $q$-Fibonacci numbers. 
\end{Remark}

\begin{proof}
These identities follow immediately from Theorem~\ref{lucas-gen} by taking
$a_k=1$, $b_k=q^{k+a}$, and  
replacing $x_n$ by 
$F_n^{(a)}(q)$. Observe that  $F_2^{(a)}(q)=1$. 

To show \eqref{garrett-lucasodd}, \eqref{garrett-lucas-square} and \eqref{qfib-lucas-alternating}, we need to additionally make calculations that are very similar to those in the corresponding identities in Example~\ref{lucas}.
\end{proof}

Next we give $q$-analogs of the Pell Number identities in Example \ref{pell}. 

\begin{Example}[$q$-Pell Numbers] Consider the $q$-Pell numbers defined as:
$P_0(q)=0$, $P_1(q)=1$, and
$$P_{n+2}(q) = \left(1+q^{n+1}\right)P_{n+1}(q) + q^{n} P_n(q).$$
 Then we have, for $n=0,1,2,\dots$:
\begin{align}
\sum_{k=1}^n q^{k}\frac{P_k(q)}{\qrfac{-q}{k+1}} 
&= \frac{P_{n+2}(q)}{\qrfac{-q}{n+1}}-1.\label{qpell-lucas1}\\
\sum_{k=1}^n \left(1+q^{2k}\right) q^{-k^2}P_{2k}(q) &= 
q^{-n^2}P_{2n+1}(q)-1.\label{qpell-lucaseven}\\
\sum_{k=0}^n \left( 1+q^{2k+1}\right) q^{-k(k+1)}P_{2k+1}(q) &= 
q^{-n(n+1)}P_{2n+2}(q).\label{qpell-lucasodd}\\
\sum_{k=1}^n \left(1+q^{k}\right) q^{-{k\choose 2}}P_{k}^{2}(q) &= 
q^{-{n\choose 2}}P_{n}(q)P_{n+1}(q).\label{qpell-lucas-square}\\
\sum_{k=0}^n (-1)^{k} q^{-{k+1\choose 2}}\qrfac{-q}{k}P_{k+2}(q) &= 
(-1)^{n} q^{-{n+1 \choose 2}}\qrfac{-q}{n+1}P_{n+1}(q).\label{qpell-lucas-alternating}\\
\sum_{k=1}^n \frac{q^{2k-1}}{1+q^{k}}
\prod_{j=1}^k \left[ \frac{1+q^j}{1+2q^j+q^{j+1}+q^{2j+1}}\right]&
 \frac{P_{k-1}(q)}{1+q}\nonumber \\
 =  1-\prod_{j=1}^n & \left[ \frac{1+q^j}{1+2q^j+q^{j+1}+q^{2j+1}}\right]
 \frac{P_{n+2}(q)}{1+q}.
 \label{qpell-lucas-2n}
\end{align}
\end{Example}
\begin{Remark} The $q$-Pell numbers were defined by  Santos and Sills \cite{santos-sills}. However, we have (for our convenience) redefined the recurrence relation a bit,  and taken the initial values $P_0(q)=0$ and $P_1(q)=1$. Our $P_{k}(q)$ is Santos and Sills' $P_{k-1}(q)$. 
Identity \eqref{qpell-lucasodd} is due to Briggs, Little and Sellers \cite[Theorem 6]{bls}. Identity \eqref{qpell-lucaseven} is also due to these authors. 
The rest of the identities appear to be new.

When $q=1$, then these identities reduce to identities for the Pell numbers in Example~\ref{pell}. 
\end{Remark}

\begin{proof}
These identities follow immediately from Theorem~\ref{lucas-gen} by taking
$a_k=\left(1+q^{k+1}\right)$, $b_k=q^k$, and  
replacing $x_n$ by 
$P_n(q)$. Note also the initial conditions of  the $q$-Pell numbers and the recurrence relation implies that $P_2(q)=1+q$. 
\end{proof}

\begin{Remark}
We can also define the {\em shifted} $q$-Pell numbers $P_k^{(a)}(q)$ by taking $a_k=1+q^{k+a+1}$ and $b_k=q^{k+a}$ and derive identities for them, as done in Example \ref{schur-q-fib}. We can also use Theorem~\ref{lucas-gen} to derive identities for the more general objects considered in \cite[\S 4]{bls}. 
\end{Remark}

Next we consider identities for the $q$-Fibonacci Polynomials. Fibonacci Polynomials satisfy the recurrence relation
$$F_{n+2}(x,y) = xF_{n+1}(x,y) + y F_{n}(x,y),$$
with initial conditions $F_0(x,y)=0$ and $F_1(x,y)=1.$ If we replace $x$ and $y$ by $1$, we obtain the Fibonacci Numbers.  Instead, if one takes $x=2$ and $y=1$, one gets the Pell Numbers. 

Another interesting special case is when $y=-1$ and $x\mapsto 2x$. Then $F_k(2x, -1)=U_{k-1}(x)$, where $U_k\equiv U_k(x)$ are the Chebychev polynomials of the second kind, see  \cite[p.~101]{aar} for the recurrence relation satisfied by $U_n$. Here $U_{-1}=0$ and $U_0=1$ and so $U_1=2x$. 

Recently, different kinds of $q$-Fibonacci Polynomials have arisen in combinatorial contexts. Here we give identities for two of them. They also extend the Chebychev polynomials of the second kind and the Pell Numbers. 
\begin{Example} 
Consider the Goyt-Sagan \cite{goyt-sagan} $q$-Fibonacci polynomials defined as:
Let $F_0^{M}(x,y, q)=0$, $F_1^{M}(x,y, q)=1$, and
$$F_{n+2}^{M}(x,y,q) = xq^nF_{n+1}^{M}(x,y,q) + yq^{n-1} F_n^{M}(x,y,q).$$
Let $F_k^M(q)\equiv F_k^M(x,y,q)$. Then we have, for $n=0,1,2,\dots$:
\begin{align}
\sum_{k=1}^n \frac{y}{x^{k+1}} q^{-{k\choose 2}-1}F_k^{M}(q) 
&= \frac{1}{x^{n+1}} q^{-{n+1\choose 2}}F_{n+2}^{M}(q)-1.\label{qFibM-lucas1}\\
\sum_{k=1}^n \frac{x}{y^{k}} q^{-k^2+3k-1}F_{2k}^{M}(q) &= 
\frac{1}{y^n} q^{-n^2+n}F_{2n+1}^{M}(q)-1.\label{qFibM-lucaseven}\\
\sum_{k=1}^n \frac{1}{y^{k}} q^{-k^2+2k}F_{2k+1}^{M}(q) &= 
\frac{1}{xy^{n}} q^{-n^2}F_{2n+2}^{M}(q)-1.\label{qFibM-lucasodd}\\
\sum_{k=1}^n \frac{1}{y^{k}} q^{-{{k\choose 2}} +k}\left(F_{k+1}^{M}(q)\right)^2 &= 
\frac{1}{xy^{n}} q^{-{{n\choose 2}}}F_{n+1}^{M}(q)F_{n+2}^{M}(q)-1.\label{qFibM-lucas-square}\\
\sum_{k=1}^n (-1)^{k} \frac{x^{k-1}}{y^{k}} F_{k+2}^{M}(q) &= 
(-1)^{n} \frac{x^n}{y^{n}} q^{n}F_{n+1}^{M}(q)-1.\label{qFibM-lucas-alternating}\\
\sum_{k=1}^n  \left(\frac{xq}{y}\right)^{k-2} \frac{F_{k-1}^{M}(q)}{\qrfac{-qx^2/y}{k}}&= 
1- \frac{x^{n-1}}{y^n}\frac{F_{n+2}^{M}(q)}{\qrfac{-qx^2/y}{n}} \label{qFibM-lucas-2n}
\end{align}
\end{Example}
\begin{Remark}
The polynomials $F_k^M(x,y,q)$ were considered by Goyt and Sagan \cite{goyt-sagan} with slightly different initial conditions. Our $F_n^M(q)$ are Goyt and Sagan's $F_{n-1}$ and 
$F_{n-1}^M(q)$ in Goyt and Mathisen \cite{goyt-mathisen}. Goyt and Sagan have proved the $x=1=y$ case of \eqref{qFibM-lucas-2n}. Set $x=1=y$ in any of these identities to obtain an identity for the relevant $q$-Fibonacci numbers. In addition, one can set $x=2$, $y=1$ to get identities for another $q$-analog of the Pell Numbers. 
\end{Remark}
\begin{proof}
These identities follow immediately from Theorem~\ref{lucas-gen} by taking
$a_k=xq^{k}$, $b_k=yq^{k-1}$, and  
replacing $x_n$ by 
$F_n^M(x,y, q)$. Note also the initial conditions of  the Goyt-Sagan $q$-Fibonacci polynomials and the recurrence relation implies that $F_2^M(q)=x$.
\end{proof}

\begin{Example} Consider the Goyt-Mathisen \cite{goyt-mathisen} $q$-Fibonacci polynomials defined as:
 $F_0^{I}(x,y, q)=0$, $F_1^{I}(x,y, q)=1$, and
$$F_{n+2}^{I}(x,y,q) = xq^{n}F_{n+1}^{I}(x,y,q) + yq^{2(n-1)} F_n^{I}(x,y,q).$$
Let $F_k^I(q)\equiv F_k^I(x,y,q)$. Then we have, for $n=0,1,2,\dots$:
\begin{align}
\sum_{k=1}^n (-1)^{k} \frac{x^{k-1}}{y^{k}} q^{-{k\choose 2}} F_{k+2}^{I}(q) &= 
(-1)^{n} \frac{x^n}{y^{n}} q^{-\frac{n(n-3)}{2}}F_{n+1}^{I}(q)-1.\label{qFibI-lucas-alternating}\\
\sum_{k=1}^n  \frac{y^2}{x^2}\left(\frac{x}{qx^2+y}\right)^{k}
q^{-\frac{(k-2)(k-5)}{2}}
F_{k-1}^{I}(q) &= 
1- \frac{x^n}{x(qx^2+y)^n}
q^{- {n\choose 2}}
F_{n+2}^{I}(q) \label{qFibI-lucas-2n}
\end{align}
\end{Example}
\begin{Remark}
The polynomials $F_k^I(q)$ were considered by Goyt and Mathisen \cite{goyt-mathisen} with slightly different initial conditions. Our $F_n^I(q)$ are 
$F_{n-1}^I(q)$ in Goyt and Mathisen \cite{goyt-mathisen}. Goyt and Mathisen proved 
the identities that follow from the first 4 identities in Theorem~\ref{lucas-gen}, see Theorem 4.4, Theorem 4.6, Theorem 4.5 and Theorem 4.7 in \cite{goyt-mathisen}. (In two of these, we have to reverse the sum to obtain the identities in the form presented in \cite{goyt-mathisen}.) Set $x=1=y$ in any of these identities to obtain an identity for the  $q$-Fibonacci numbers that correspond to the Goyt-Mathisen $q$-Fibonacci polynomials. In addition, one can set $x=2$, $y=1$ to get identities for another $q$-analog of the Pell Numbers. 
\end{Remark}
\begin{proof}
These identities follow immediately from \eqref{lucas-gen-alternating} and \eqref{lucas-gen-div2n} by taking
$a_k=xq^{k}$, $b_k=yq^{2(k-1)}$, and  
replacing $x_n$ by 
$F_n^I(x,y, q)$. Note also the initial conditions of  the Goyt and Mathisen's $q$-Fibonacci polynomials and the recurrence relation implies that $F_2^I(q)=x$, and $F_3^I(q)=qx^2+y$. These two  values show up in formula \eqref{qFibI-lucas-2n}. 
\end{proof}

Cigler \cite{cigler-fib} considered $q$-Fibonacci polynomials defined as:\\
$F_0^{C}(x,y, q)=0$, $F_1^{C}(x,y, q)=1$, and
$$F_{n+2}^{C}(x,y,q) = xF_{n+1}^{C}(x,y,q) + t(yq^{n}) F_n^{C}(x,y,q).$$
Here $t$ is a general function. Cigler considers special cases $t(y)=y$ and $t(y)=y/q$.  The $x=1$ and $t(y)=y$ case is $\mathcal{S}_n(y,q)$ considered by Andrews \cite{and4}.  One can easily write down identities that follow from Theorem \ref{lucas-gen} for these polynomials, which will generalize the identities in Example \ref{schur-q-fib}.  Cigler's polynomials are also motivated by work done by Carlitz \cite{carlitz2}, and thus we can get identities for Carlitz's $q$-Fibonacci polynomials too.

Finally, before closing this section, we note another generalization of Fibonacci numbers given by Ismail \cite{ismail-fib}. These are defined as:
$F_0(\theta)=0$, $F_1(\theta)=1$, and
$$F_{n+2}(\theta) = 2\sinh\theta F_{n+1}(\theta) + F_n(\theta).$$
Once again, one can take the $a_k=2\sinh \theta$ and $b_k=1$ in Theorem \ref{lucas-gen} and derive identities for $F_k(\theta)$. The identity that follows from \eqref{lucas-gen-square} appears in Ismail \cite{ismail-fib}.

\section{Concluding Remarks}

After studying Euler's proof of the Pentagonal Number Theorem,  Andrews \cite{and1} emphasizes 
\lq\lq how valuable it is to study and understand the central ideas behind major pieces of mathematics produced by giants like Euler". 

We  cannot agree more!
As we have seen, an elementary identity appearing in Euler's  proof of his Pentagonal Number Theorem can be used to prove a wide variety of identities. 
Its power is demonstrated by our EZ proof of Jackson's $q$-analog of Dougall's result. This is because all  the key summation theorems  of hypergeometric and $q$-hypergeometric series (for terminating and non-terminating series) are special cases of the $q$-Dougall sum. 
Similarly,  Theorem~\ref{lucas-gen} is able to unify many identities for many combinatorial sequences.

Further study of Euler's Telescoping Lemma should be fruitful. A careful survey of Fibonacci identities can possibly yield generalizations to sequences  of the kind considered here.  In addition, Theorem~\ref{lucas-gen} applies to orthogonal polynomials as well, since they satisfy a three-term recurrence relation, but we have hardly considered them here.

\end{document}